\documentclass{article}
\usepackage{color}
\usepackage{amsmath,amssymb,amsthm,amsfonts}
\usepackage{esint}
\usepackage{hyperref}
\usepackage{wrapfig}
\usepackage{color}
\makeatletter

\newcommand{\Rmnum}[1]{\expandafter\@slowromancap\romannumeral #1@}
\makeatother

\newtheorem{thm}{Theorem}[section]

\newtheorem{definition}{Definition}[section]

\newtheorem{lemma}[thm]{Lemma}

\newtheorem{proposition}[thm]{Proposition}

\usepackage[numbers,sort&compress]{natbib}
\usepackage{graphicx}
\usepackage{caption}
\usepackage[font=footnotesize,labelformat=simple]{subcaption}
\usepackage{tikz}
\usepackage{multirow}
\usepackage{multicol}

\usepackage{lineno}
\usepackage{lscape}
\usepackage{threeparttable}
\usepackage{authblk}

\date{\today}

\title{An Optimal Switching Approach for Bird Migration}

\author{{Jiawei Chu$^{a,b,}$}
~,~ King-Yeung Lam$^{b,}$\thanks{Corresponding author:lam.184@osu.edu; KYL is
supported by National Science Foundation Grant DMS-2325195}~,~
Boyu Wang$^{c}$,~~and~
Tong Wang$^{b}$
\\
\small \emph{$^{a}$Department of Applied Mathematics, The Hong Kong Polytechnic University,}\\
\small \emph{ Kowloon,  Hong Kong.}\\
\small \emph{$^{b}$Department of Mathematics, The Ohio State University,}\\
\small \emph{ Columbus, 43210, Ohio, USA.}\\
\small \emph{$^{c}$School of Mathematics, Fudan University,}\\
\small \emph{ Shanghai 200433, China}}

\begin{document}

\maketitle

\begin{abstract}

    Bird migration is an adaptive behavior ultimately aiming at optimizing survival and reproductive success. We propose an optimal switching model to study bird migration, where birds' migration behaviors can be efficiently modeled as switching between different stochastic differential equations. For individuals with perfect information regarding the environment, we implement numeric methods to see the expected payoff and corresponding optimal control. For individual with only partial information of the environment, we combine the finite difference method and stochastic simulations to investigate the change of the bird's optimal strategy. Based on biological backgrounds, we characterizing the optimal strategies of birds under different scenarios and these behaviors depend on the specific assumptions of the model.
\end{abstract}

\section{Introduction}\label{sec:1}

\subsection{Biological setting and general questions}

Migration is a seasonal movement of organisms between seasonal favorable regions. Bird migration, in particular, is remarkable for their long distances, geological barriers involved and predation risks. Each year, billions of birds driven by the optimal living habitats take on journeys that span up to thousands of kilometers, across ocean and continents. Avoiding \textit{intraspecific competition} and exploiting crucial resources in seasonal favorable regions
for breeding were usually considered as two main elements for evolution of bird migration \cite{salewski2007evolution}. Various bird species show a wide range of complex migratory behaviors. For example, different species show both high and low route repeatability, consistency and variation in duration of migration \cite{vardanis2016consistency}, and different migratory strategies between spring and fall migration \cite{nilsson2013differences}. The route and schedule of migration are affected by many factors, including the presence of geological barriers, predation risk \cite{ydenberg2007effects}, 
global warming \cite{Jenni2003timing}, access to climate information at the terminal site \cite{carneiro2020linking}, and population density dependence \cite{gourley2015age}. Although bird migrations exhibit complex behaviors, they all serve the same purpose: to optimize survival probability and reproductive success. {Stochastic optimal control, an effective tool for modeling population dynamics in biological/ecological systems \cite{lenhart2007optimal,yoshioka2019stochastic}, is usually employed to model bird migration \cite{alerstam2011optimal, houston1999models, mangel2015stochastic, purcell2007factors}.} 


\medskip

\subsection{Stochastic optimization approach}

{Optimal control theory is {ubiquitous in} theoretical and applied ecology. It predicts how an individual agent can navigate a set of risks and reward to maximize a payoff functional, which includes the reproductive gain and the survival rate in a migration process. In general, a type of behavior maximizing the payoff functional will tend to increase in frequency in the population if they have a heritable component through genetics or learning. The payoff functional is often stochastic (e.g., depending on the mean arrival time at the destination following a diffusion process), and is expressed in terms of an expectation. In such a case, stochastic dynamic programming (SDP) can be used to reformulate the stochastic optimization problem into a set of deterministic partial differential equations,  to be solved backward in time. Here we follow the convention in ecology and resource management to refer to both the model itself and the solution approach as SDP.

SDP has been widely used in population dynamics \cite{lenhart2007optimal,yoshioka2019stochastic}, {particularly in} the study of bird migration \cite{alerstam2011optimal, houston1999models, mangel2015stochastic,purcell2007factors}.
{Conventional bird migration models based on SDP \cite{purcell2007factors} are usually discretized in space, using a large number of variables to characterize the migration process, which often makes them analytically intractable. For continuous-time models, solving a stochastic optimization problem reduces to analyze the underlying Hamilton-Jacobi-Bellman (HJB in short) equation \cite[Chapter 9]{Fleming-2006-book} by using analytical method (e.g.,\cite{zhu-analytical}) or numerical techniques (e.g., \cite{Parzani-mumerical}), and the optimal feedback control may be obtained as a byproduct.} 

\subsection{Our objective}

The objective of this paper is to propose an optimal switching model of bird migration with stopover decision based on SDP. The optimal switching model \cite[Chapter 5]{Pham-2009-book} is a decision-making framework in which individuals switch between two or more diffusion processes to maximize objective function. In our model, the migratory birds switch between three states: ``detour", ``direct flight", and ``resting at a stopover site", each is governed by a separate stochastic differential equation (see Section \ref{sec:2}). 
The optimal switching strategy in our model (location and timing to switch between states) {is derived by} {studying} a deterministic system of {HJB} variational inequalities { through both analytical and numerical framework, see Section \ref{sec:3} and \ref{sec:4}, respectively. Finally, in Section \ref{sec:5}, we apply this optimal switching model to answer several specific biological questions by characterizing the optimal strategies of birds under different scenarios.}

\subsection{Applications of optimal control in our model}

There are many interesting biological questions arising from animal migration. 
The first question we discuss in this paper is the effect of the deterioration of stopover site. Climate change can dramatically reshape the migration behaviors of birds \cite{alerstam2011optimal, ward2009change,walther2002ecological}. 
 Extensive research has demonstrated the profound influence of temperature increase on bird migration choice \cite{ward2009change,butler2003disproportionate}. Furthermore, 
 the loss of wetland due to increasing temperature \cite{bickford2010impacts} can greatly reduce the number of options of migratory shorebirds which depend heavily on the stopover wetlands along the coast for food supplies. Hence, it is natural to ask that
 \begin{itemize}
 \item Can we quantify and analyze the change in their migratory payoff as the quality of major stopover sites decline, and predict the corresponding change in the migration strategy?
 \end{itemize}
We shall explore this question and present our simulation results in Subection 5.1.

Another important factor influencing the migration process is the weather dynamics \cite{carneiro2020linking}. Instead of focusing on the stopover site, we will place more emphasis on the effect of the stochastic weather condition happening at the terminal site. 
However, in nature, birds can anticipate environmental change based on past experience and perception \cite{carneiro2020linking,carneiro2019faster} but only to a limited degree. In particular, they may not be aware of every spontaneous change of climate, especially the condition at the terminal site when it is far away. To this end, we consider the effect of imperfect versus perfect information about the weather condition at the terminal site. Based on the principle of adaptation, we assume that birds have formed a stable strategy according to the average condition of the weather dynamics at the terminal site, which is based on the past experience. But the actual quality of the terminal site may contain unpredictable daily fluctuation due to stochasticity, or be shifted in time due to global climate change. These factors lead to potential mismatch of the perceptual quality of the terminal site versus its actual quality that the bird experiences that particular year. 

As such, we will additionally address these two questions in Subection 5.2:
\begin{itemize}
\item Will only knowing the information about the averaged condition at the terminal site decrease the expected payoff significantly if we compare with the same situation but with perfect information? 
\item How will the migratory payoff be impacted if the optimal arrival date (or peak green-timing) deviates significantly from past experience, due to the global climate change \cite{robertson2024decoupling}?
\end{itemize}

\section{Mathematical modeling}\label{sec:2}

In this section, we formulate the migration dynamics and the objective function to be optimized. 
\subsection{Basic setting}
\begin{wrapfigure}{r}{.3\textwidth}
\begin{center}
\includegraphics[ width=\linewidth]{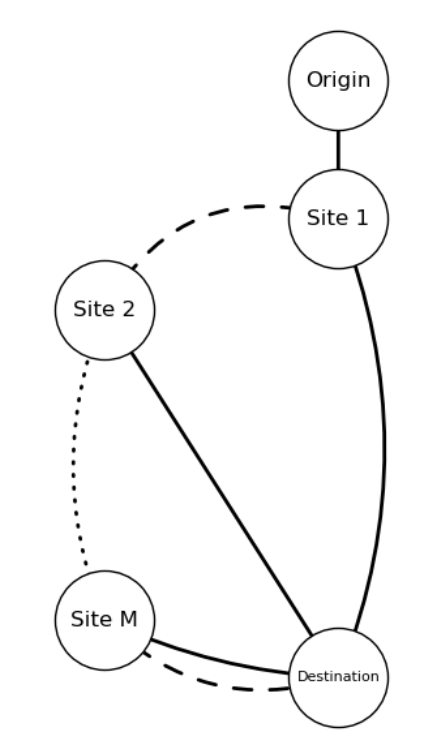}
        \caption{Geographical setting of the migration route.}
        \label{fig:1b}
\end{center}

\end{wrapfigure}
We model the bird migration process as a stochastic optimal switching problem. In this model, individuals switch between different states (representing detour flight, direct flight and waiting) where they are governed by a different drift-diffusion process. A value function $V$ representing the expected payoff function when optimal control is exercised. To evaluate $V$, we derive a system of variational inequalities and prove that $V$ is the unique viscosity solution of the system of variational inequalities (see Sections \ref{sec:2.4} and \ref{sec:3}). 

The essential idea of the model can be presented by a simple qualitative example of a migrating Pacific brant, which travels more than 5000 km from breeding ground at Izembek Lagoon on the Alaska Peninsula to their wintering ground in Southern Baja California or mainland Mexico \cite{purcell2007factors}. Each individual has a $T$-day period to complete the migration to its terminal site. For simplicity, we consider that the bird can not survive unless it reaches the terminal site.

For the model, we model the Pacific as a linear interval $[0,L]$ (with $x=0$ being the starting location, and $x=L$ being the destination). A typical individual can {switch} between the two states: ``detour flight" or ``wait" as it moves along the interval $[0,L]$, where ``wait" is only possible at one of the $M$ stopover sites. Additionally, as individual departs from one of the $M$ stopover sites, it may choose to embark on a direct flight across the Pacific ocean towards the terminal site (Baja Peninsula), which is modeled by the state ``direct flight" which has a greater mean speed and variability comparing to the ``detour flight". See Figure \ref{fig:1b} for an illustration. Instead of using a two-dimensional  graph, we model the choice of direct flight by the fact that an individual cannot swtich back to ``detour flight" or ``wait" once it adopts a ``direct flight", and the higher mean speed also reflects that fact that the physical distance of a direct flight (crossing the ocean) is smaller than a detour flight (along the coast).



The whole switching process is shown in Figure \ref{fig:2}. Here we label the $i$-th stopover sites by the interval $[L_i-\epsilon, L_i+\epsilon]$.  For x $\in [0, L_1-\epsilon)$, where the bird still migrate within the continent and do not reach the ocean, the individual bird continues in the state ``detour flight" until it reaches the first stopover site $[L_1-\epsilon, L_1+\epsilon]$ (representing the first staging area near the ocean). At this point, it can choose either to detour along the coast or to take direct flight over the ocean. If it adopts a direct flight, then it switches to the respective diffusion process which does not terminate until it arrives at the destination. Otherwise, it continues to the next stopover site and the process repeats itself. 
         

\begin{figure}[h]
    \centering
    \includegraphics[width=0.7\textwidth]{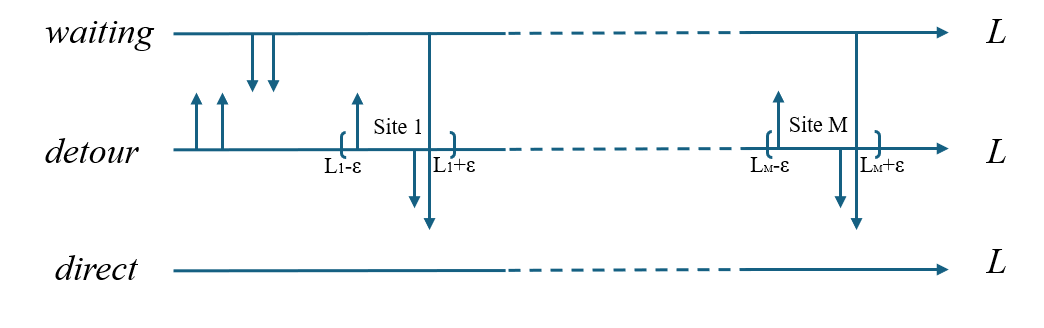}
    \caption{The three states of an individual (waiting/detour/direct) with switching control}
    \label{fig:2}
\end{figure}

\subsection{Migration dynamics}

We translate the bird migration into mathematical framework for optimal switching problem. Denote $L$ as the  distance between starting point and terminal site, and $T<\infty$ as the terminal time for bird migration, then we consider an optimal switching problem on a finite horizon.  Fix a probability space $(\varOmega, \mathcal{F}, P)$ with a filtration {\color{black}$\mathbb{F} = (\mathcal{F}_{s})_{s \geqslant t}$} satisfying the {{\em usual}} conditions (see, e.g. \cite{Fleming-2006-book,Pham-2009-book}). The regime space 
\begin{equation*}
\Pi= \Pi_m = \{1,2,....,m\}
\end{equation*}
(with $m=3$ in our case), and describes the state that a bird is staying in. Then denote $(t,x,i)\in [0,T]\times[0,L]\times\Pi$ as the initial conditions, which describes the exact position $x$ of an individual bird at time $t$ and regime $i$.

A switching control is a double sequence $\alpha = (\tau_n, \iota_n)_{n \geqslant 1}$, where $(\tau_n)$ are an increasing sequence of stopping times, $\tau_n \rightarrow \infty$ as $n\rightarrow\infty$ a.s., denote the decision on “when to switch”; $\iota_n$ are $\mathcal{F}_{\tau_n}$-measurable valued in $\Pi$, represent the new value of the regime during time $[\tau_n, \tau_{n+1})$.  The set of all switching controls is denoted as $\mathcal{A}$. Given the initial regime $i$, we can define the controlled switching process among different states as follows:
\begin{equation}\label{control}
\pi_s := \pi(s,i)=\sum_{n \geq 0} \iota_n \chi_{[\tau_n,\tau_{n+1})}(s), s \geq t, \quad \pi_{t^{-}} = i,
\end{equation}
where $\chi_{[\tau_n,\tau_{n+1})}$ is the indicator function, defined by
$$\chi_{B}
:=\chi_{B}(y)=\begin{cases}
1, &\text{if}~y\in B,\\
0, &\text{otherwise};
\end{cases}
$$
and we set $\tau_0:=t$, $\iota_0:=i$ and $t\in[0,T)$ is any fixed. Then we use a Markov process $\{X^{t,x,i}_s\}_{s\geq t}$ to represent the whole migration, starting from position $x$ and the state $i$ at an initial time $t$, which satisfies the following stochastic differential equation
\begin{equation}\label{migration}
\begin{split}
    dX_s^{t,x,i}&=\sum_{j=1}^3v_j\chi_{\{j=\pi_s\}}ds+\sum_{j=1}^3\sqrt{2\mu_j}\chi_{\{j=\pi_s\}}dB_s,\ \ \text{for}~0\leq t\leq s<T.
\end{split}
\end{equation}
For the remainder of this article, let
the regime state $i=1$ represent the state of detour flight, $i=2$ represent the state of direct flight, and $i=3$ represent the state of waiting, and let $B_{s}$ be a 1-dimensional Brownian motion on the filtered probability space $(\varOmega,\mathcal{F}, \mathbb{F}, P)$ satisfying the usual conditions. The constants $v_j\geq 0$ and $\mu_j\geq 0~(j=1,2, 3)$ represent the velocity and volatility of the regime, respectively, where $v_i$ and $\mu_i$ are constants satisfying
$$
v_2 > v_1 >0 = v_3\quad \text{ and }\quad \mu_1 > \mu_2 > 0 = \mu_3, 
$$
which is motivated by the fact that directly flight shortens the route distance to destination ($v_2>v_1$) while increases stochasticity ($\mu_2>\mu_1$) owing to weather variations, such as wind direction, wind velocity, etc.

Consequently, for any initial condition $(t,x,i)$ and any given control $\alpha\in\mathcal{A}$, the migration dynamics can alternative be written as:
\begin{equation}\label{X12}
dX^{t,x,i}_s = v_{\iota_n} ds + \sqrt{2\mu_{\iota_n}}dB_s, \ \ \pi_s=\iota_n, \forall~s\in[\tau_n,\tau_{n+1}), \ \ n\geq 1,
\end{equation}
where $\pi_s$ is defined in \eqref{control}. When $\iota_n\ne 3$, 
from \cite[Theorem V.38]{Protter-2005-book}, we get that  there exists a unique strong solution valued in $[0,L]$ to the standard SDE \eqref{X12} for any $(t,x)\in[0,T]\times [0,L]$, denote it by $X^{t,x,i}_s$; when $\iota_n=3$, 
we set, for each
$(t,x)\in[0,T]\times [0,L]$, 
\begin{equation}\label{X3}
 X^{t,x,i}_s= X^{t,x,i}_{\tau_n} \quad \text{ for all }s \in [\tau_n, \tau_{n+1}].
\end{equation}

\subsection{Objective function}
The objective function is a function to be maximized by the individual bird through its migration process subject to the migration dynamics formulated in the previous subsection. Here are some key components for the construction of objective function.

\begin{itemize}
\item[(H1)] {\bf Running and terminal reward functions.} 
For $j=1,2,3$, we let $f_j(t,x)=f(t,x,j):~[0,T]\times [0,L]\times\Pi\rightarrow \mathbb{R}$ and  $g_j(t,x)=g(t,x,j):$ $[0,T]\times [0,L]\times \Pi\rightarrow \mathbb{R}$ 
be, respectively, the running and terminal reward functions. We assume that they are uniformly Lipschitz continuous in $(t,x)$.

\item[(H2)] {\bf Mortality/Discount rate.} For $j=1,2,3$, {we} denote the mortality rate associated with birds' action by $\beta_j(t,x)=\beta (t,x,j):~[0,T]\times [0,L]\times\Pi\rightarrow \mathbb{R}$ and assume that it is Lipschitz continuous. In reality, the mortality rate for direct flight is considered higher than that for detour flight, reflecting the increased challenges birds face over the sea, such as the inability to rest periodically and greater exposure to adverse weather conditions.
\item[(H3)] {\bf Switching cost.} We denote the cost to switch from regime $k$ to $j$, by $h_{kj}(t,x):~[0,T]\times [0,L]\rightarrow \mathbb{R}$ for $k,j\in\{1,2,3\}$. For each $k,j$, we assume $0\leq h_{kj}(t,x)\in C^{1,2}([0,T]\times[0,L])$, where the equality holds iff $k=j$. Moreover, suppose that the function $h_{kj}(t,x)$ satisfy
\begin{equation}\label{hc1}
h_{kj}< h_{kq}+h_{qj}, \ \  \text{ for }q\ne k,~j, \text{ such that } h_{kq}+h_{qj}<+\infty,
\end{equation}
\noindent and 
\begin{equation}\label{hc2}
h_{2j}= +\infty, \ \ j=1,~3.
\end{equation}
The condition 
\eqref{hc1} means that switching in two steps via an intermediate regime $q$ is more costly than in one step from regime $k$ to $j$; \eqref{hc2} implies that the bird cannot switch from direct flight to detour or waiting state.  
\end{itemize}

The arrival time at $x=L$ is a hitting time of the process:
%
$$\tau=\inf\{s\geq t: X^{t,x,i}_s=L\}.$$
The objective function give the expected benefit, given the current state $(t,x,i)\in [0,T]\times[0,L]\times{\Pi}$ and subject to a given control $\alpha=(\tau_n,\iota_n)_{n\geq 1}\in\mathcal{A}$:
\begin{equation}\label{J}
\begin{split}
   J_i(t,x,\alpha)=E^{t,x,i}\bigg[\int_{t}^{\tau\wedge T}
		& e^{-\beta_{\pi_s}(s, X_{s}^{t,x,i}) (s-t)} f_{\pi_s}(s,X_{s}^{t,x,i})ds -\sum_{ t\leq\tau_n<\tau\wedge T} e^{-\beta_{\pi_s}(\tau_{n}, X_{s}^{t,x,i})(\tau_n-t)} {h_{l_{n-1}l_{n}}}(s,X_s^{t,x,i}) \\
		& + e^{-\beta_{\pi_s}(\tau, X_{\tau}^{t,x,i}) (\tau-t)} g_{\pi_s}(\tau,L)\chi_{\{\tau\leq T\}}\bigg],
\end{split}
\end{equation}
where $X_s^{t,x,i}$ is the unique strong solution to \eqref{migration} subject to the given control process $\alpha$.  Under proper assumptions 
{(see (H1)-(H3))}
$J_i(t,x,\alpha)$ is well-defined for any initial condition $(t,x,i)$ and any control process $\alpha\in \mathcal{A}$. 

Finally, the value function of the optimal control problem is defined as the supremum of the objective function $J_i(t,x,\alpha)$ over the set of all admissible control processes $\mathcal{A}$:
\begin{equation}\label{eq: optimal_control}
\begin{split}    V_i(t,x)=\sup_{\alpha \in \mathcal{A}} J_i(t,x,\alpha), \ \ {(t,x,i)\in[0,T]\times[0,L]\times \Pi.}
\end{split}
\end{equation}
By \eqref{J}, it is natural to set the terminal condition
$$
 V_i(T,x)=0 \quad \text{ for }x\in[0,L),\\
$$
and boundary conditions
$$(V_i)_x(t,0) = 0, \quad V_i(t,L)=g_i(t,L),  \quad \text{ for }t\in[0,T],
$$
which says that the diffusion process is reflected at the boundary $x=0$ and is absorbing at $x=L$.

The goal of this optimal switching problem is to find a strategy $\alpha^*=(\tau^*_n,\iota^*_n)_{n\geq 1}\in\mathcal{A}$ to achieve the maximum of \eqref{eq: optimal_control}, $\alpha^*$ is called the optimal strategy. A key ingredient of the Bellman's approach of stochastic optimal control is the dynamic programming principle, which we now recall. 
\begin{lemma}[Stochastic dynamic programming principle (SDPP)]\label{sdpp}
For any $(t,x,i)\in [0,T]\times[0,L]\times\Pi$, we have
\begin{equation*}
    \begin{aligned}
		V_i(t,x) = \sup_{\alpha \in \mathcal{A}} E^{t,x,i} \bigg[ \int_{t}^{\tau  \wedge \theta } 
		& e^{-\beta_{\pi_s}(s, X_s^{t,x,i}) (s-t)} f(s,X_{s}^{t,x,i},\pi_s)ds -\sum_{t\leq \tau_n <\tau \wedge \theta} e^{-\beta_{\pi_s}(\tau_{n},X_s^{t,x,i}) (\tau_{n}-t)} h_{l_{n-1}l_{n}}(s,X_s^{t,x,i}) \\
		&+e^{-\beta_{\pi_s}(\tau, X_s^{t,x,i}) (\tau-t)} g_i(\tau,L)\chi_{\{\tau<\theta\}}+ e^{-\beta_{\pi_s}(\theta, X_s^{t,x,i}) (\theta-t)} V_{i}(\theta,X_{\theta}^{t,x,i})\chi_{\{\theta\leq{\tau\wedge T}\}} \bigg],
    \end{aligned}
\end{equation*}
 where $\theta\in [t, T]$ is any stopping time, possibly depending on $\alpha\in\mathcal{A}$.
\end{lemma}
\begin{proof}
    See, e.g. \cite[Lemma 4.2]{Fleming-2006-book}.
\end{proof}

\subsection{Hamilton-Jacobi-Bellman variational inequality}
\label{sec:2.4}

To find the optimal control $\alpha^*$, the crucial step is to show, based on the the dynamic programming principle (see Lemma \ref{sdpp}) and the concept of viscosity solutions (see Definition \ref{dvs}), that $(V_i(t,x))_{i\in\Pi}$ is the unique viscosity solution to the following system of Hamilton-Jacobi-Bellman (HJB in short) variational inequality with terminal and lateral conditions:
\begin{equation}\label{HJB12}
    \begin{cases}
        \min[-\partial_tV_{i} + \beta_i V_i - \mathcal{L}_iV_i - f_i, V_i - \max_{i \neq j}(V_j - h_{ij})] = 0, &(t,x)\in (0,T)\times(0,L),\\
     {(V_i)_x}(t,0)=0, &t\in[0,T],\\
     V_i(t,L)=g_i(t,L),   &t\in[0,T],\\
      V_i(T,x)=0,  &x\in[0,L),
    \end{cases}
\end{equation}
where  $i\in\{1,2,3\}$ and  
\begin{equation}\label{e.a.L}
\mathcal{L}_i \phi:=v_i\phi_x+\mu_i\phi_{xx} 
\end{equation}
(see details in Proposition \ref{cvs} in Section 3).

This system of PDE facilitates the identification of the associated optimal control $\alpha^*$. Precisely, 
{for} any $i\neq j$, we define the switching region to be the closed set
\begin{equation}\label{SIJ}
    \mathcal{S}_{ij} = \{(t,x)\in (0,T)\times (0,L): V_i(t,x) = (V_j - h_{ij})(t,x)\},
\end{equation}
{which means that if an individual in state $i$ is located at location $x$ and time $t$, such that $(t,x) \in \mathcal{S}_{ij}$, then it will switch from regime $i$ to regime $j$. Define
\begin{equation}\label{SI}
\mathcal{S}_i = \cup_{j\neq i} \mathcal{S}_{ij} \quad \text{ and }\quad \mathcal{C}_i = [0,T]\times [0,L] \setminus \mathcal{S}_i.
\end{equation}
It follows from \cite[Lemma 4.2]{Pham-2007-chapter} that
$$\mathcal{S}_{i} = \{(t,x)\in (0,T)\times (0,L): V_i(t,x) = \max_{j \neq i}(V_j - h_{ij})(t,x)\},$$
then
$$\mathcal{C}_i= \{(t,x)\in [0,T]\times [0,L]: V_i(t,x) > \max_{j \neq i}(V_j - h_{ij})(t,x)\}.$$
By Proposition \ref{os}, we get that $\mathcal{S}_i$ is the region where changing the regime $i$ is optimal. If} $(t,x) \in \mathcal{C}_i$, then it is optimal to stay in regime $i$ at least for a short time (Proposition \ref{os}). Hence, $\mathcal{C}_i$ is called the continuation region. 


Moreover, by \eqref{hc1}, the switching regions are disjoint, so that $$\mathcal{S}_{ij}\subset  \mathcal{C}_j, \ \ j\ne i\in\Pi.$$
 Roughly speaking, an individual who has just switched  from regime $i$ to $j$ does switch immediately from $j$ to another regime. 
When $i=2$, { assumption (H3)implies} that there is no switch, this gives $\mathcal{S}_2=\emptyset$ and $\mathcal{C}_2=[0,T]\times[0,L]$.

Consequently, the optimal strategy for an individual bird/agent can be fully characterized by the switching regions $\mathcal{S}_{12},\mathcal{S}_{13},\mathcal{S}_{31},\mathcal{S}_{32}$, as given in \eqref{SIJ}.
\section{Theory}\label{sec:3}
In this section, we first give the definition of vicosity solution, and then we shall show that $(V_i(t,x))_{i\in\Pi}$ is the unique viscosity solution to \eqref{HJB12}
(see Propostion \ref{cvs}). This result enables us to get the associated optimal control $\alpha^*$ (see Proposition \ref{os}).
\vspace{2mm}

 Let $t\in[0,T]$ and $\mathcal{O}\subset\mathbb{R}^{n}~(n\geq 1)$ be an open connected domain with closure $\overline{\mathcal{O}}$ and smooth boundary $\partial\mathcal{O}:=\overline{\mathcal{O}}\verb|\| \mathcal{O}$ satisfying  the exterior ball condition. Let $\partial \mathcal{O}=\Gamma_1\cup\Gamma_2$ with $\Gamma_1$ being open set and $\Gamma_2$ being closed set.  Consider the following general system of HJB variational inequality with mixed boundary conditions:
{
\begin{equation}\label{Gsystem}
\begin{cases}
     \min[-\partial_tu_{i} + F_i(t,x,u_i,Du_i,D^2u_i), u_i - \max_{i \neq j}(u_j - h_{ij})] = 0, &(t,x)\in (0,T)\times\mathcal{O},\\
     \partial_\nu u_i(t,x)=\phi_i(t,x), &(t,x)\in(0,T)\times \Gamma_1,\\
     u_i(t,x)=d_i(t,x), &(t,x)\in[0,T)\times \Gamma_2,\\
     u_i(T,x)=\psi_i(x),& x\in \overline{\mathcal{O}},
\end{cases}
\end{equation}
}
 where $F_i(t,x,u_i,Du_i,D^2u_i):=\beta_i u_i - \mathcal{L}_iu_i - f_i$ {with $\mathcal{L}_i$ defined in \eqref{e.a.L}} and $\partial_\nu u_i:=\frac{\partial u_i}{\partial \nu}$ with $\nu$ denoting the unit outward normal vector of boundary $\Gamma_1$. We suppose that
\begin{itemize}
\item [(Ha)] Nonnegative functions $\beta_j(t,x)$ and $f_j(t,x)$: $[0,T]\times \overline{\mathcal{O}}\times\Pi\rightarrow \mathbb{R}$ are uniformly Lipschitz continuous in $(t,x)$; the switching cost $0\leq h_{kj}(t,x)\in C^{1,2}([0,T]\times\overline{\mathcal{O}}):~[0,T]\times \overline{\mathcal{O}}\rightarrow \mathbb{R}$ for $k,j\in\Pi$ satisfies \eqref{hc1}, where ``='' holds iff $k=j$.
\item [(Hb)] For each $i\in\Pi$, $\phi_i(t,x)$ and $d_i(t,x)$ are continuous on $[0,T]\times\Gamma_1$ and $[0,T]\times\Gamma_2$, respectively. The function $\psi_i(x):\overline{\mathcal{O}}\times\Pi\rightarrow \mathbb{R}$ is continuous such that 
$$\psi_i(x)\geq \max_{i \neq j}(\psi_j(x) - h_{ij}(T,x)).$$
\end{itemize}
 
Under assumptions (Ha)-(Hb), we can introduce the concept of viscosity solutions. We first recall some fundamental notations. Given a locally bounded function $u(t,x):[0,T]\times\overline{\mathcal{O}}\rightarrow~\mathbb{R}$ (i.e., for all $(t,x)\in [0,T]\times \overline{\mathcal{O}}$, there exists a compact neighborhood $V_{(t,x)}$ of $(t,x)$ such that $u$ is bounded on $V_{(t,x)}$). We recall that its lower-semicontinuous envelope $u_*$ and upper-semicontinuous envelope $u^*$ on $[0,T]\times \overline{\mathcal{O}}$ (see \cite[pp.267, Definition 4.1]{Fleming-2006-book}) are given  respectively by
\begin{equation*}
\begin{split}
 &u_*(t,x):=\liminf_{(s,y)\rightarrow (t,x)} u(s,y):=\sup\limits_{r>0}\text{inf}\{u(s,y)\in[0,T]\times \overline{\mathcal{O}}\cap B_r(t,x)\},\\
 &u^*(t,x):=\limsup_{(s,y)\rightarrow (t,x)}u(s,y):=\inf\limits_{r>0}\text{sup}\{u(s,y)\in[0,T]\times \overline{\mathcal{O}}\cap B_r(t,x)\},
\end{split}
\end{equation*}
i.e., $u_*$ (resp. $u^*$) is the largest (resp. smallest) lower-semicontinuous function (l.s.c.) below (resp. upper-semicontinuous function (u.s.c.) above) $u$ on $[0,T]\times \mathcal{O}$. Observe that $u(s,y)$ is continuous at $(t,x)\in[0,T]\times \mathcal{O}$ if and only if $u(t,x)=u_*(t,x)=u^*(t,x)$

 Next, we give the definition of viscosity solution as Definition 4.2 and Remark 4.2 in \cite[ pp.267]{Fleming-2006-book} or Definition 4.2.1 in \cite{Pham-2009-book}, which is equivalent to \cite[Definition 2.3]{2022-AMO-Olofsson} or \cite[Definition 3.3]{Hamadene-2023-JMAA} by applying Lemma 4.1 in \cite[pp.211]{Fleming-2006-book}.

\begin{definition}[Viscosity solution]\label{dvs}
  Let $u_i(t,x)$ ($i\in \Pi$) be locally bounded. Then function $(u_i)_{i\in\Pi}$ is a viscosity subsolution of \eqref{Gsystem}, if for all $i\in \Pi$ and any $W\in C^{1,2}([0,T]\times \overline{\mathcal{O}})$, 
   \begin{itemize}
    \item[(i)] {$u_i^*(t,x)-W (t,x)$ admits a  maximum point $(\bar{t},\bar{x})\in[0,T)\times\overline{\mathcal{O}}$, and
    $$\min[-\partial_tW + F_i(t,x,u_{i}^*,DW,D^2W), u_{i}^* - \max_{i \neq j}(u_{j}^* - h_{ij})]\leq 0 $$
    holds for all $(\bar{t},\bar{x})$;}
    \item[(ii-1)] {$u_i^*(t,x)-W (t,x)$ admits a  maximum point $(\bar{t},\bar{x})\in[0,T)\times\Gamma_1$ and
    $$\min[-\partial_tW + F_i(t,x,u_{i}^*,DW,D^2W), u_{i}^* - \max_{i \neq j}(u_{j}^* - h_{ij})]\wedge [\partial_\nu u_i^*-\phi_i)]\leq 0 $$
    holds for all $(\bar{t},\bar{x})$;}
    \item[(ii-2)] $$u_{i}^*(t,x)-d_i(t,x)\leq 0 $$
    holds for all $(t,x)\in[0,T)\times\Gamma_2$;
    \item[(iii)] $$u_i^*(T,x)\leq \psi_i(x) $$
    holds for all $x\in\overline{\mathcal{O}}$;
   \end{itemize}
   A viscosity supersolutions are defined analogously by replacing ($u_k^*$, $\leq$, $\wedge$) with ($(u_{k})_*$, $\geq$, $\vee$), respectively, where $a\wedge b=\min\{a,b\}$ and $a\vee b=\max\{a,b\}$. A function is a viscosity solution of \eqref{Gsystem} if it is both a viscosity subsolution and viscosity supersolution \eqref{Gsystem}.
\end{definition}

The comparison principle stated below is derived from \cite[Theorem 2.4]{2022-AMO-Olofsson}; see also the discussion in \cite[pp.75]{Pham-2009-book}.
\begin{lemma}[Comparison principle]\label{cp}Let $t\in[0,T]$ and $\mathcal{O}\subset\mathbb{R}^{n}~(n\geq 1)$ be an open connected set with closure $\overline{\mathcal{O}}$, and smooth boundary $\partial\mathcal{O}:=\overline{\mathcal{O}}\verb|\| \mathcal{O}$ satisfying the exterior ball condition. Let $\partial \mathcal{O}=\Gamma_1\cup\Gamma_2$ with $\Gamma_1$ being open set and $\Gamma_2$ being closed set.  Assume that (Ha)-(Hb) hold. If $(u_i)_{i\in\Pi}$ and $(U_i)_{i\in\Pi}$ are respectively a visosity subsolution and viscosity supersolution of \eqref{Gsystem}, then for each ${i\in\Pi}$, 
$$ u_i\leq U_i, \ \ \ \forall~(t,x)\in[0,T)\times\overline{\mathcal{O}}.$$
\end{lemma}

\vspace{2.5mm}
{Using Definition \ref{dvs} and Lemma \ref{cp}, we shall show that $(V_i(t,x))_{i\in\Pi}$ is the unique viscosity solution to \eqref{HJB12}.}
\begin{proposition}[Existence and uniqueness]\label{cvs}
  For each $i\in\Pi$, the value function $V_i\in C([0,T)\times[0,L))$ is the unique viscosity solution to {\eqref{HJB12} under the Definition \ref{dvs}}.
\end{proposition}
\begin{proof} Take $\mathcal{O}=(0,L)$, $\Gamma_1={0}$, $\Gamma_2={1}$, $\phi_i(x)=0$, $d_i(t,x)=g_i(t,L)$ and $\psi_i(x)=0$, assumptions (H1)-(H3) indicate that (Ha)-(Hb) hold. Proceeding with the same procedures as the proof in \cite[Theorem 5.2]{2016-KHARROUBI} (or see e.g., \cite[Theorem 5.3.2]{Pham-2009-book}) and \cite[Theorem 3.2]{Hamadene-2023-JMAA}, we can show that for each $i\in\{1,2,3\}$, the value function $V_i$ is the viscosity solution to \eqref{HJB12}
under the Definition \ref{dvs}.

Next, we prove that $(V_i)_{i\in\Pi}$ is the unique viscosity solution. We will show the result by contradiction. Assume on the contrary that $(U_i)_{i\in\Pi}$ and $(V_i)_{i\in\Pi}$ are both the viscosity solutions to \eqref{HJB12}.
Then by definitions of viscosity solution and lower/upper-semicontinuous envelopes, $U_i^*$ and $V_i^*$ (resp. $(U_i)_*$ and $(V_i)_*$) are the viscosity subsolutions (resp. viscosity supersolutions). Hence, with assumptions (H1)-(H3) in hand, Lemma \ref{cp} implies that $U_i^*\leq (V_i)_*$ and $V_i^*\leq (U_i)_*$ on $[0,T)\times[0,L)$, which together with the facts $(U_i)_*\leq U_i^*$ and $(V_i)_*\leq V_i^*$ (which hold by construction), gives that
$$U_i= (U_i)_* =U_i^*=(V_i)_*=V_i^*=V_i \ \ \text{in}\ \ [0,T)\times [0,L).$$
This means that $V_i\in C([0,T)\times [0,L))$ is the unique viscosity solution to \eqref{HJB12} under the Definition \ref{dvs}.
\end{proof}
The assumption (H3) indicates that no switching occurs if the initial regime state $i=2$. When the initial regime state $i\ne 2$, \cite{Hamadene-2009-SIAM} and \cite{Ludkovski-2005} give the following result regarding the optimal strategy. 
\begin{proposition}[Optimal switching strategy]\label{os}
Let assumptions (H1)- (H3) hold and the initial regime state $i\in\{1,3\}$.  Definie the sequence $\mathbb{F}$-stopping times $(\tau^*_n)_{n\geq 1}$ as follows:
\begin{equation*}
\begin{split}
  & \tau^*_1:=\inf\{s\geq t: V_i=\max_{j\ne i}(V_j-h_{ij})\}\wedge \tau\wedge T,\\
  & \tau^*_n
  \begin{cases}
  =\inf\{s\geq \tau^*_1: V_{\pi_{\tau^*_{n-1}}}=\max\limits_{k\ne \tau^*_{n-1}}(V_k-h_{\pi_{\tau^*_{n-1}}k})\}\wedge \tau\wedge T,& \text{if}\ \ \pi_{\tau^*_{n-1}}\ne 2;\\
  \nexists, &\text{if}\ \ \pi_{\tau^*_{n-1}}= 2,
  \end{cases} 
  \ \ \text{for~any}~n\geq 2,
\end{split}
\end{equation*}
where,
\begin{itemize}
\item $\pi_{\tau_1^*}=\sum\limits_{k\in\Pi} k\chi_\{\max\limits_{j\ne i}(V_j-h_{ij})(\tau_1^*,X_{\tau^*_1}^{t,x,i})=(V_k-h_{ik})(\tau_1^*,X_{\tau^*_1}^{t,x,i})\}$
\item for any $n\geq1$, $t\geq \tau_n^*$, $V_{\pi_{\tau_n^*}}(t,x)=\sum\limits_{k\in\Pi}V_k(t,x)\chi_{\{\pi_{\tau_n^*}=k\}}$;
\item for any $n\geq 2$ and $\pi_{\tau_{n-1}^*}\ne 2$, $\pi_{\tau_{n}^*}=l$ on the set
$$\Bigg\{\max\limits_{j\ne \pi_{\tau_{n-1}^*}}(V_j-h_{\pi_{\tau_{n-1}^*}j})(\tau_n^*,X_{\tau^*_n}^{t,x,i})=(V_l-h_{\pi_{\tau_{n-1}^*}l})(\tau_n^*,X_{\tau^*_n}^{t,x,i})\Bigg\}$$
with $h_{\pi_{\tau_{n-1}^*}j}(\tau_n^*,X_{\tau^*_n}^{t,x,i})=\sum\limits_{k\in\Pi}h_{kj}(\tau_n^*,X_{\tau^*_n}^{t,x,i})\chi_{\{\pi_{\tau_{n-1}^*}=k\}}.$
\end{itemize}
Then the strategy $\alpha^*=(\tau^*_n,\iota^*_n)_{n\geq 1}$ is optimal, i.e., $J_i(t,x,\alpha^*)\geq J_i(t,x,\alpha)$ for all $\alpha\in\mathcal{A}$.
\end{proposition}

\section{Numerical Methods}\label{sec:4}


This section devoted to the numerical simulation results. 
To go beyond the standard optimal control framework, we explore the effect of partial vs perfect information in terms of the terminal reward function, representing the payoff an individual received if it survives the trip and arrives at the destination at a certain time $t$.  {\color{black}For simplicity, we focus on the case where the terminal reward is independent of the regime state (detour vs direct flight on arrival at $x=L$).} This reward function typically depends on climatic and biotic conditions such as availability of resources and competition. We assume it is a given function of $t$ for simplicity. 

Let $g(t)$ be the perceived terminal reward as a function of arrival time $t$ at $x=L$, which the individual leverages to optimize its switching strategy. This perceived terminal reward represents the ``best guess" of the individual concerning the condition at the terminal site $x=L$, which is inhabited from the past experience of the collective experience. Next, let $G(t)$ be the actual terminal reward. Consider the following two cases 
\begin{itemize}
    \item {(Perfect information)}  \quad  $g(t) = G(t)$;
  \item  \text{(Partial information)} \quad $g(t) \neq G(t).$
\end{itemize}

For the case of perfect information, we will use the variational PDE to calculate the value function $V(t,x)$ and obtain the optimal control given by switching regions, and the migratory payoff given the optimal control for individuals started from $(0,0)$, which will be represented by $V(0,0)$. This is the subject of the next subsection.

\subsection{Perfect information}\label{pi}
In Sections \ref{sec:2} and \ref{sec:3}, we derived the deterministic PDEs of the value functions for the optimal control problem $\eqref{eq: optimal_control}$. For clarity, we recall that:
\begin{equation}
    \left\{
    \begin{aligned}
        &\min [-\partial_t V_i + \beta V_i - \mathcal{L}_iV_i - f_i, V_i - \max_{j \neq i}(V_j - h_{ij})] = 0 & (t,x) \in [0, T) \times (0,L), \\
        & {(V_i)_x}(t,0) = 0 & t \in [0,T], \\
        &V_i (t,L) = G(t) & t \in (0,T], \\
        &V_i (T,x) = 0 & x \in [0,L),
    \end{aligned}
    \right.
 \label{e.4.8}
 \end{equation}
 where $\mathcal{L}_i$ is given in \eqref{e.a.L}.
Then, we use a finite difference scheme to solve the coupled system \eqref{e.4.8}. We give a brief introduction to the process below.

\begin{itemize}
    \item [] {\bf Step 1}: The $(t,x)$ domain is discretized into a lattice with step sizes $\Delta t$ and $\Delta x$. Denote $V_{i,m,n} = V_i(m \Delta t, n \Delta x)$, then the parabolic part of the equation can be discretized according to the implicit Euler scheme:

    \begin{equation}
        \begin{aligned}
            -&\frac{V_{i,m+1,n}-V_{i,m,n}}{\Delta t} - v_i\frac{V_{i,m,n+1}-V_{i,m,n}}{\Delta x} \\
            &- \mu_i \frac{V_{i,m,n+1}- 2 V_{i,m,n}+V_{i,m,n-1}}{{\Delta x}^2} + \beta_i V_{i,m,n} - f_{i,m,n} = 0.
        \end{aligned}
        \label{eq: discrete_form}
    \end{equation}

    \item [] {\bf Step 2}: Let the solution $(V_{i,m+1,n})_{1\leq i \leq 3, 1 \leq n \leq N_x}$ be given for the $(m+1)$-th time step, we solve \eqref{eq: discrete_form} for the solution at the $m$-th time step and denote the result by $(\hat{V}_{i,m,n})_{1\leq i \leq 3, 1 \leq n \leq N_x}$.

\item[] {\bf Step 3}: We update again to account for the variational inequality arising from the switching control by the equation below,
    \begin{equation}
        V_{i,m,n} = \max_{j \neq i} (\hat{V}_{i,m,n}, \hat{V}_{j,m,n}-h_{ij}).
        \label{eq: variational inequalities}
    \end{equation}

\end{itemize}

By taking the procedures above, we can obtain the optimal payoff value $V(t,x)$ for individuals starting at $(t,x)$, which estimates the largest expected value individual can obtain by executing the optimal control.
\label{perfect information}

\subsection{Switching region}\label{sr4}

In this subsection, we show how to represent the optimal strategy through the computation of value functions above. To achieve this goal, we recall the definitions of switching region and continuation region as below:
\begin{equation*}
\text{(Switching region)}\qquad     \mathcal{S}_{i} = \{(t,x)\in [0,T]\times [0,L]: V_i(t,x) =\max_{j\ne i}(V_j - h_{ij})(t,x)\}.
\end{equation*}
The complement set of $\mathcal{S}_{i}$ denoted by
\begin{equation}
\text{(Continuation region)}\qquad       \mathcal{C}_i = \{(t,x)\in [0,T]\times [0,L] : V_i(t,x) > \max_{j \neq i}(V_j - h_{ij})(t,x)\}.
    \label{e.4.11}
\end{equation}
We also denote the closed set representing the switching region from regime $i$ to regime $j$ by  $\mathcal{S}_{ij}$:
\begin{equation}
    \mathcal{S}_{ij} = \{(t,x)\in [0,T]\times [0,L]: V_i(t,x) = (V_j - h_{ij})(t,x) \quad \text{for some} \quad j \neq i\}.
    \label{e.4.12}
\end{equation}
With these notations, Proposition \ref{os} shows that the diffusion process for an individual adopting the optimal switching strategy can be fully characterized by a controlled diffusion process where an individual switch between different diffusion processes according to its physical location $(t,x)$ and its current state $i$, as given by \eqref{e.4.11} and \eqref{e.4.12}. To numerically demonstrate this, we follow the following steps: 
\begin{itemize}
    \item First, applying the numerical methods described in sub-section \ref{pi}, we obtain the value functions $V_i(t,x)$ for any fixed initial condition $(t,x,i)$. 
    \item Then, we compare $V_i(t,x)$ for different regimes $i$ by taking switching cost $h_{ij}$ (cost of switching from regime i to regime j) into consideration to draw the switching regions and continuing regions.
\end{itemize}

\begin{figure}[h!]
\centering
\includegraphics{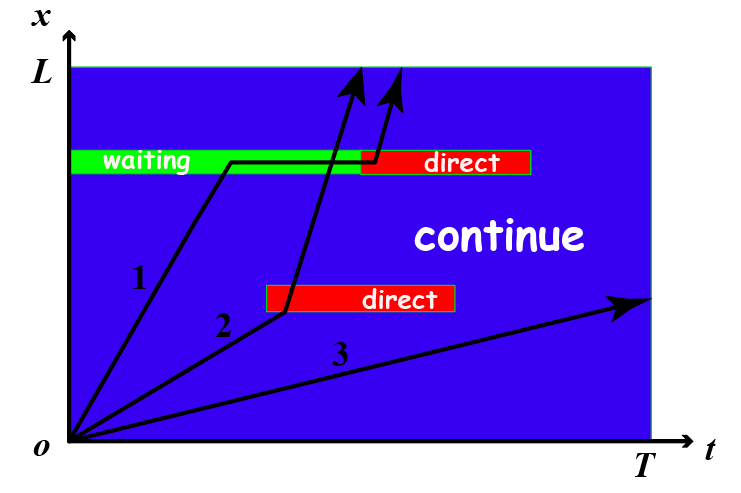}
\caption{(An illustration of 
switching regions)  Consider three different realizations of the diffusion process (representing three distinct individuals). In path 1, the individual arrives at the green region, so it is optimal for the individual to switch to the waiting state $i=3$ and stay there until it is time to switch to ``direct flight" (which is the diffusion region $i=2$). In the second path, the individual arrives at the red region at which point it is optimal to adopt direct flight immediately. (Note that once the individual adopts direct flight, no further switching is possible so it does not matter that it enters the green region. The above diagram only concerns individuals adopting switching modes $i=1$ or $i=3$.) Finally, in path 3, the individual does not enter the green or red region, so it is optimal for the individual to stay at diffusion mode of ``detour flight" $i=1$.}
\label{fig:switching_region_2}
\end{figure}


\vspace{2mm}
\subsection{Implement Optimal Control}

In this section, we discuss how to implement the optimal control $\alpha^*$ determined by the switching regions derived in the subsection \ref{sr4} with stochastic simulations.
Consider the stochastic processes that model the movement of individuals over time, governed by \ \eqref{X12}. To analyze this process, we discretize $(t,x)$ domain into a lattice with step sizes $\Delta t \ \text{and } \Delta x$, with absorbing boundaries (e.g. the process will be terminated once individual touch the boundary point $x=L$). Then individual movements will be modeled as a random walk with drift on the discretized domain with the transition probabilities $p_\ell, p_r$  (and the probability $1=p_\ell-p_r$ to stay put). We first fix spatial step size and then obtain the time step size according to spatial step to maintain numerical stability. After that, we derive for each regime $i$ the transition probabilities $p_\ell$, $p_r$ for each regime. For simplicity, we present the formulas for $i=1$, as the formulas for the other diffusion regime is similar.
\begin{equation*}
u_t = \mu_1 u_{xx} + v_1 u_{x},
\end{equation*}
let 
\begin{equation*}
\begin{split}
&\Delta x = h,\ \ \Delta t = \frac{\delta}{\mu_1+\mu_2} h^2, \\ &\frac{p_l+p_r}{2} = \frac{\mu_1}{\mu_1+\mu_2},\\
&p_l-p_r = \frac{v_1 \Delta t}{\Delta x} = v_1 \times\Delta x \times \frac{\Delta t}{(\Delta x)^2} = \frac{v_1 \Delta x}{\mu_1+\mu_2}.
\end{split}
\end{equation*}
Then one gets  

\begin{equation*}
\begin{split}
&u_t = \frac{p_l+p_r}{2} \times \frac{(\Delta x)^2}{\Delta t}\times u_{xx}+(p_l-p_r)\times\frac{\Delta x}{\Delta t}\times u_x,\\
&p_l = \delta \bigg[\frac{\mu_1}{\mu_1+\mu_2} - \frac{v_1 h}{2(\mu_1+\mu_2)}\bigg],\ \ p_r = \delta \bigg[\frac{\mu_1}{\mu_1+\mu_2} + \frac{v_1 h}{2(\mu_1+\mu_2)}\bigg].
\end{split}
\end{equation*}}

Given the optimal control, individuals governed by the stochastic diffusion process \eqref{X12} will start from $(t,x,i)$ and follow the random walk with transitional probabilities obtained above until one of the following cases happen:
\begin{itemize}
    \item \textbf{Case 1:} $t=T$ or the individual hits the terminal site $x=L$.
    \item \textbf{Case 2:} The first time $t=\tau_1$ when the individual enters $(t,x) \in \{S_{ij}\}_j$ for some $j\neq i$.  
\end{itemize}

In the first case, the diffusion process will terminate. In the second case, the individual will continue the diffusion process for $t \geq \tau_1$ with new regime state $j$ (and hence new values of $(p_\ell,p_r)$ for the approximate random walk), until one of the above scenarios happen again. 
Figure \ref{fig:switching_region_2} is an illustrative example of stochastic processes described above. As the individual hits the switching region, switching time $\tau_n$ and state after switching $\iota_n$ are recorded. When case 1 happens, a series of switching time and states $(\tau_n, \iota_n)_{n \geqslant 0}$ can be obtained.

By repeating the above simulation multiple times, 
the quantitative statistics of bird migration can be obtained, which will be further analyzed in the following numerical investigation into specific biological problems.
\label{sec:stochastic simulation}

{Worth to mention that by our way of discretization, the spatitial step $\Delta x$ will be bounded by $\frac{2\mu}{v}$. For numerical stability, and the time step $\Delta t$ = $\frac{\delta}{\mu_1+\mu_2} (\Delta x)^2$, where $\delta$ is a fixed small number, can be extremely small, which result in a large number of computations will be required for a single simulation. To resolve the problem,  we combine $n$ time steps as one large step by first compute the probability distributions on the state space of steps moving forward $[-n, -n+1, \dots, 0, 1, \dots, n-1, n]$ steps and only check for update of regime for each individual every $n$ steps, with $n$ being a reasonable small integer so that the simulation still follows the optimal control strictly. With this new algorithms implemented, we were able to reduce the amount of computations needed by $n$ times.}

\subsection{Partial information}

In reality, organisms make decisions based on perceived information. It is not possible in general, even for highly mobile organisms such as birds, to obtain perfect information of the environment, which depends on factors such as climate change, weather dynamics and so on. To model the partial information perceived by birds, we use a weighted average of historical payoffs, denoted by $g$ to reflect that information was based on the past migratory experiences of birds. Thus, we assume the bird possesses partial information $g$ as the average value of the perfect information $G$. An extreme case is $g = \fint G(t)$.

We propose a numerical scheme to model the optimal migration behavior under partial information as follows.

\begin{itemize}
    \item [] {\bf Step 1}: Compute $V^{partial}_i(t,x)$ as optimal control value with perceived terminal reward $g$,
    \begin{equation}
        \begin{aligned}
    		V^{partial}_i(t,x) = \sup_{\alpha \in \mathcal{A}} E^{t,x,i} \bigg[ \int_{t}^{\tau \wedge T}
    		& e^{-\beta_{i}(s, x) (s-t)} f_{i}(s,X_{s}^{t,x,i}) \, ds -\sum_{\tau_n \leqslant \tau} e^{-\beta_{i}(\tau_{n}, x) (\tau_{n}-t)} h_{l_{n-1},l_{n}} \\
    		& + \chi_{\left\{ \tau < T \right\}} e^{-\beta_{i}(\tau, x) (\tau-t)} g(\tau) \bigg],
        \end{aligned}
        \label{eq: partial_information}
    \end{equation}
    to obtain resulting optimal control $(\hat{\tau}_n, \hat{\iota}_n)$ is expressed in terms of the switching region described in section 4.2. Note that the control is suboptimal with respect to the actual terminal payoff $G$.

    \item [] {\bf Step 2}: Using the stochastic simulations introduced in subsection \ref{sec:stochastic simulation}, where the actions of individual are governed by the optimal strategy obtained in Step 1 based on partial information $g$, we can record a family of arriving time {\color{black}$(\tau^n)_{1 \leq n \leq N}$} at $x=L$. Then we use the Monte-Carlo method to compute the difference in expectations to corresponding to the perceived terminal payoff versus the actual terminal payoff function $G(t)$.
    \begin{equation}
        E^{t,x,i}(g) \approx \frac{1}{N}\sum_{n=1}^{N} g{\color{black}(\tau^n)}, \quad E^{t,x,i}(G) \approx \frac{1}{N}\sum_{n=1}^{N} G{\color{black}(\tau^n)},
    \end{equation}
    which implies
    \begin{equation}
        E^{t,x,i}(g-G) = E^{t,x,i}(g) - E^{t,x,i}(G).
    \end{equation}
    \item [] {\bf Step 3}: Update the expected payoff value $V^{actual}_i$, accounting for the difference between $g$ (the bird's perception of terminal reward) and $G$ (the actual terminal reward). To obtain the actual payoff
    \begin{equation}
        \begin{aligned}
    		V^{actual}_i(t,x) &= E^{t,x,i} \bigg[ \int_{t}^{\tau \wedge T}
    		 e^{-\beta_{i}(s, x) (s-t)} f_{i}(s,X_{s}^{t,x,i}) \, ds -\sum_{\tau_n \leqslant \tau} e^{-\beta_{i}(\tau_{n}, x) (\tau_{n}-t)} h_{l_{n-1},l_{n}} \\
    		& \qquad \qquad + \chi_{\left\{ \tau < T \right\}} e^{-\beta_{i}(\tau, x) (\tau-t)} G(\tau) \bigg]\\
           &= V_i^{partial}(t,x) - E^{t,x,i}(g-G) 
        \end{aligned}
        \label{eq: perfect_information}
    \end{equation}
    under the control $\alpha$ in the form of switching regions obtained from $V^{partial}_i(t,x)$. 
\end{itemize}


From these process above, we can also observe how the switching behavior of the bird differ when the environment is different from their perceptions. That is, the difference of switching regions and optimal controls can be observed under the guidance of different environmental settings with various running rewards, terminal rewards, predation risks, etc.

\section{Numerical Results}
\label{sec:5}
 In this section, we address the questions posed in the introduction through numerical simulations. In
this section, we use colormaps to illustrate the value of the value function $V_1(t,x)$ as well as the optimal control. Precisely, for Figures \ref{fig:V(0,0) against lamda}, \ref{fig:Optimal Control2}, \ref{fig:Optimal Control3}, \ref{fig:regime1(region)_imperfect(case1)} and \ref{fig:regime1(simulation)_imperfect(case1)}, 
the red (resp. green) region represents the switching regions $\mathcal{S}_{12}$ (resp. $\mathcal{S}_{13}$), corresponding to the switching from detour to direct flight (resp. detour to waiting at stopover site). The remaining region is $\mathcal{C}_{1}$ (where the bird remains in the detour/slow-flight state). In the latter region, we use color map to indicate the dependence of the value function $V_i$ on the time and space variables.

\subsection{Deterioration of stopover site}
\label{sec:Deterioration}
Global warming frequently lead to the deterioration of stopover sites, rendering one or more of the stopover sites unsuitable for migrating birds. 
This subsection aims to study the impact of the deterioration of stopover sites in our model on the migration dynamics of shorebirds. 

In accordance with the geographical set up laid out in Figures \ref{fig:1b} and \ref{fig:2} in the introduction, we denote $X_k$ as the first instance of $x$ value such that the $k$-th stopover site is available. Then define $M_k = \{x \in [0,L]: x \in [X_k, X_k+\epsilon_k]\}$, as the locations of stopover sites in which the individual who are not in direct flight (recall that individual in direct flight has to continue in its journey until reaching $x=L$) can choose to switch to any other states. We denote the set $\mathbb{M} = \{ M_k \mid k = 1, \ldots, n \}.$ Specifically, we define 
$$
\epsilon_k = \frac{L}{50},~~X_1 = \frac{L}{3} - \frac{\epsilon_k}{2},~~X_2 = \frac{4L}{9} - \frac{\epsilon_k}{2},~~X_3 = \frac{3L}{5} - \frac{\epsilon_k}{2},~~ \text{ and }~ ~X_3 = \frac{5L}{6} - \frac{\epsilon_k}{2}.$$ In addition, we define $M_0 = [0, X_1)$ to be the positions of wintering site where individual can choose to switch between regime 1 and 3 (detour and waiting), but they can't switch to regime 2 (direct flight) since it haven't reached the sea. We fix the values of other parameters as in Table \ref{parameter table1}.

\begin{table}[h]
\centering
\begin{threeparttable}
\caption{Parameters}
\begin{tabular}{|p{3cm}|c|p{8cm}|}
\hline
\textbf{Parameter} & \textbf{Baseline Value} & \textbf{Description} \\ \hline
$L$ & 5301 km\tnote{a} & Direct distance from Izembek to west coast of Mexico \\ \hline

$v_1$ & 373.33 km/day\tnote{b} & Detour flight velocity \\ \hline
$v_2$ & 560 km/day\tnote{c} & Direct flight velocity \\ \hline
$v_3$ & 0 & Velocity of waiting \\ \hline
$T$ & 72 days\tnote{d} & Total number of days covered \\ \hline
$\mu_1$ & 150\,km$^2$/(day) & Volatility associated with detour migration \\ \hline
$\mu_2$ & 250\,km$^2$/(day) & Volatility associated with direct migration \\ \hline
$\mu_3$ & 0 & Volatility associated with direct migration \\ \hline
$\beta_1$ & 0.00116/day\tnote{e} & Mortality risk for detour migration \\ \hline
$\beta_2$ & 0.002/day\tnote{e} & Mortality risk for direct migration \\ \hline
$\beta_3$ & 0.0002/day\tnote{e} & Mortality risk at stop site 1 day \\
\hline
$n$ & 3 & Number of stopover sites available \\ \hline

{$f_1$} & {0} & {Reward for regime 1: detour flight}\\\hline
{$f_2$} & {0} & {Reward for regime 2: direct flight}\\\hline
{$f_3$} & {$\displaystyle \begin{cases} 
0.00156 \tnote{d} & \text{wintering site} \\
0.000907 \tnote{d}& \text{1st stopover site} \\
0.00551 \tnote{d}& \text{2nd stopover site} \\
0.00288 \tnote{d}& \text{3rd stopover site} \\
0 & \text{otherwise}
\end{cases}$} & {Reward for regime 3: waiting}\\\hline
$h_{ij}$ $i,j \in (0,1,2,3) $ & 0.06 & Switch Cost\\\hline
\end{tabular}
\label{parameter table1}
\begin{tablenotes}
\item[a] According to \cite{dau1992fall}, most Brent fly {approximately} 5301 km during winter migration. Here we consider spring migration {and assume direct distance remains the same.}
\item[b]  
{We assume the detour velocity is lower than the direct velocity to account for the longer migration distance.}
\item[c] 
{According to \cite{dau1992fall}, Brent migrate at a speed of about 80 kph, we assume on average birds migrate 7 hours per day, which translates to 560 km per day. }
\item[d] Spring migration of black brants starts around mid-Jan {and ends} around early-April\cite{Lewis2020Brant}. So spring migration {lasts approximately} 72 days. 

\item[e]
\cite{ward1997seasonal} estimates that the average annual survival rate for black brants is 0.84, so we approximate spring migration total mortality rate around 0.08, then averaged mortality rate is around 0.00116 per day.
\item[f] According to \cite{moore2004staging}, the food abundance level 
is reported as follows: Baja California, Mexico (1,800 ha), Humboldt Bay (1,045 ha), Willapa Bay and Grays Harbor (6,650 ha), Boundary Bay  (3,320 ha), and Izembek Lagoon, Alaska (16,000 ha).  The reward is calculated by use the ratio of stopover site food level and Izembek food level divide by total migration duration.

\end{tablenotes}
\end{threeparttable}
\end{table}

Since availability of resources or breeding opportunity is sensitive to time, individuals that arrive too early (e.g., before eelgrass is exposed from winter ice) or too late (e.g., after eelgrass has been depleted) would compromise their ability to accumulate endogenous reserves \cite{smith2012trends}. Thus, we use a Gaussian function \begin{equation}
G(t) = \exp\left(-\frac{\left(t - \frac{T}{2}\right)^2}{2\left(\frac{T}{4}\right)^2}\right)
\label{perfect terminal}
\end{equation} to represent the terminal reward as a function of the arrival time at $x=L$ in \eqref{e.4.8} which is maximized at the time $t=T/2$. We set the number of stopover sites to be 3 to represent the opportunity to switch in the early, middle, or late stages of the migration process; see Figure \ref{fig:stopover Sites}.
\begin{figure}[htbp!]
\centering
\includegraphics[width=0.5\textwidth]{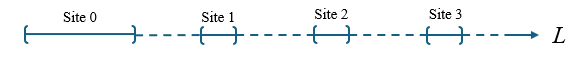}
\caption{stopover sites}
\label{fig:stopover Sites}
\end{figure}
\medskip
Then we impose different rewards for individuals to engage in staging at these three sites with the second staging site presenting the greatest reward, as shown in the Table \ref{parameter table1}. The running reward functions are denoted by $f_{i,3}(x,t)$ for $i=1,2,3$, where the `$3$' in the subscript means that the individual is in regime `$3$' (i.e. the 'waiting' regime).
In the next section, we model the deterioration of the quality of site 2.

\subsection*{\normalsize Modeling the effect of deteriorated stopover sites}

The earlier snowmelt and onset of growing season in Arctic due to global warming \cite{hupp2018spring,stone2002earlier,bhatt2010circumpolar} can result in earlier availability of food for black brants.    Here, we only account for the impact of global warming to stopover site, and use $\lambda \in [0,1]$ to denote the level of deteriation. Precisely, the running reward function at the second stopover site is given as follows.

\begin{equation}\label{e.1218.1}
    f^{\lambda}_{2;3}(x,t) = (1-\lambda) f_{2;3}(x,t), \\ \text{for} \  \lambda \in [0,1].
\end{equation}
We choose to vary the quality of the second stopover site, as it is adopted by the optimal stopover sites for all individuals given by the optimal control as explained below. If there is no deterioration in stopover sites, we obtain the optimal controls as shown in Figure \ref{fig:optimal control1} by applying the methods introduced in subsection \ref{perfect information}. 
Note that in Figure \ref{fig:optimal control1}, although it is optimal to switch to staging at some $(t,x) \in M_3$, individuals started at $(0,0)$ will always arrive at $(t,x) \in M_2$ first, at which the individual switches to direct flight and thus skipping over stopover site three. Hence, the third stopover site was never utilized if we only consider individuals started at $(0,0)$, even though the green region in Figure \ref{fig:optimal control1} indicates that it is optimal to stop there conditioned on an individual reaching it in detour flight state.



Next, we study the effect of stopover sites' deterioration on (i) migratory payoff and optimal control; (ii) optimal stopover region. 

\subsubsection{Effect on migratory payoff and optimal control}  
In Figure \ref{fig:V(0,0) against lamda}, we plot the change in migratory payoff $V(0,0)$ (conditioned on individual starting at $(t,x)=(0,0)$) as the running benefit at the staging site deteriorated (i.e., $\lambda$ increases from $0$ to $1$ in \eqref{e.1218.1}).
\begin{figure}[htbp!]
    \centering
    \begin{minipage}{0.4\textwidth} 
        \includegraphics[width=\linewidth]{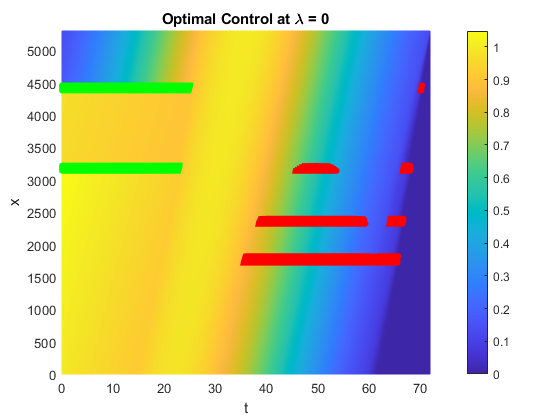}
       \caption{Optimal control at $\lambda = 0$}
       \label{fig:optimal control1}
    \end{minipage}
    \hfill 
    \begin{minipage}{0.4\textwidth} 
        \includegraphics[width=\linewidth]{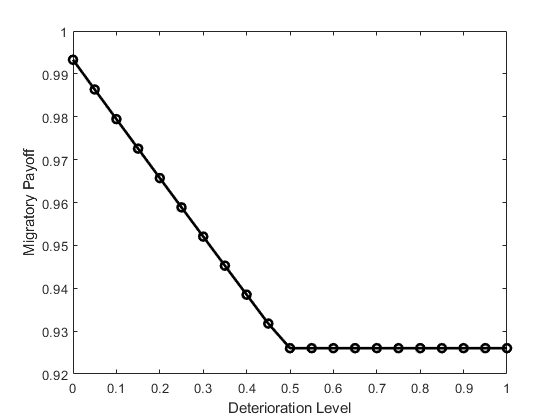}
        \caption{Migratory payoff against $\lambda$}
        \label{fig:V(0,0) against lamda}
    \end{minipage}
\end{figure}
It is observed in Figure \ref{fig:V(0,0) against lamda} that $V(0,0)$ decreases as $\lambda$ increases in $[0,0.5]$. This implies a lower migratory payoff as deterioration worsens at a stopover site. We also observed that the slope approaches $0$ after $\lambda = 0.5$. Later on, we will see that this is due to the complete abandonment of the stopover site for $\lambda$ large.

 To explore the effects on optimal control, we observe the optimal controls at $\lambda = 0.5$ (see Figure  \ref{fig:Optimal Control2}) and $\lambda = 0.55$ (see Figure \ref{fig:Optimal Control3}).
\begin{figure}[h!]
    \centering
    \begin{minipage}{0.4\textwidth} 
        \includegraphics[width=\linewidth]{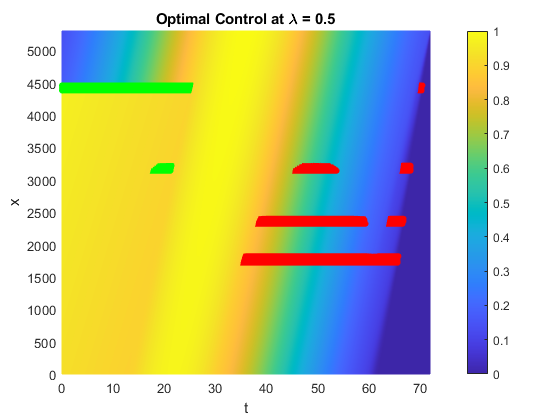}
       \caption{Optimal control at $\lambda = 0.6$}
       \label{fig:Optimal Control2}
    \end{minipage}
    \hfill 
    \begin{minipage}{0.4\textwidth} 
        \includegraphics[width=\linewidth]{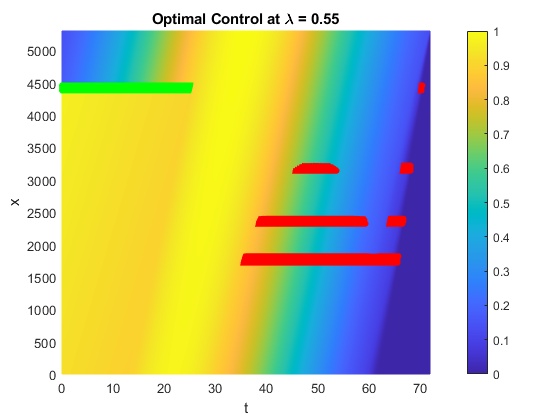}
        \caption{Optimal control at $\lambda = 0.65$}
        \label{fig:Optimal Control3}
    \end{minipage}
\end{figure}
Compared to  Figure \ref{fig:optimal control1} ($\lambda=0$, i.e, without deterioration), the switch region at stopover site 2 contracted significantly indicated by Figure \ref{fig:Optimal Control2}. whereas the the region vanished completely in Figure \ref{fig:Optimal Control3}, indicating that it is optimal for the migratory population to abandon stopover site 2. Since the stopover site is no longer utilized, further deterioration no longer impact the migratory payoff, which explains the change of slope in Figure \ref{fig:V(0,0) against lamda}. Note that there is a slight occurrence of switch region at switch site 0, which in our model is the point where bird reach the sea from wintering site. Thus, the occurrence of switch region at site 0 can be understand as a delay of departure.

\vspace{2.5mm}
\subsubsection{Effect on optimal switching regions}
To better investigate the behavior of individuals under the change of optimal control,  we utilize stochastic simulation method introduced in sub-section \ref{sec:stochastic simulation}. Here, we mainly observe the change in behavior at stopover site as the optimal control varies, which can be studied explicitly by observing the change of duration of staying at each stopover site $\overline{L_k}$. By a stochastic simulation of migrations with total $m$ individuals with optimal control given by \eqref{e.4.11}
and \eqref{e.4.12}, $\overline{L_k}$ can be obtained by 
\begin{equation}
\overline{L_k} = \frac{\sum_{j=1}^{m}\text{D}_{jk}}{\sum_{j=1}^{m} \mathbb{I}(D_{jk} > 0)}, \ \  
\end{equation}
where $D_{jk}$ denotes the {duration of staying} of the $j$-th individual at the $k$-th stopover site. Then, $\overline{L}^d_j$, the average length of stay at a particular stopover site after incorporating deterioration, can be obtained by repeating the previous simulation procedures. Taking a different value of $\lambda$, the direct impact of a different level of deterioration on the length of stay at a stopover site can be studied as shown in Figure \ref{fig:LOS against lamda}, in which the length of stay at each stopover site was displayed.
\begin{figure}[h!]
\centering
\includegraphics[width=0.5\textwidth]{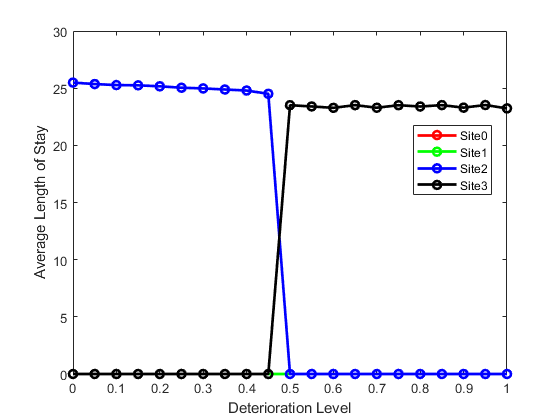}
\caption{Average length of stay at stopover site aginst $\lambda$}
\label{fig:LOS against lamda}
\end{figure}

Based on our numerical results in Figure \ref{fig:LOS against lamda}, we have three main observations as follows. \begin{itemize}
    \item For the stopover site $k=2$, beyond the critical deterioration level $\lambda=0.5$, the average duration of staying drops levels off at $0$, which explains why $V(0,0)$ is no longer senstive to further increase in deterioration level $\lambda$ for $\lambda > 0.5$ (see Figure \ref{fig:V(0,0) against lamda}). In \cite{lehikoinen2017counteracting}, it is reported that migratory birds were forced to choose sub-optimal stopover sites due to the degradation of optimal stopover site, which result in less successful migrations.
    
    \item When $\lambda \leq 0.5$, there is no clear correlation between the length of stay at stopover site 2 and deterioration level $\lambda$, differing from the negative correlation observed in \cite{smith2012trends}, where migratory individuals were observed to spend less time at the stopover site gradually. 
    
    \item As the deteriorated level $\lambda$ increases, the duration of staying at stopover site 3 increases, which indicates the importance of it improved as site 2 was modeled to be deteriorated. stopover site 0 and 1 was never be utilized, which is due to the running rewards at stopover site 0 and 1 are much smaller compared to other stopover sites and the distance between them and other sites are small relative to migratory velocities, so there is no advantage staying at stopover site 0 and 1.
\end{itemize}

\vspace{2.5mm}

Next, we introduced the modified terminal reward 
\begin{equation}
G^\lambda(t, \gamma) = \exp\left(-\frac{\left(t - \left(\frac{T}{2} + \gamma \cdot T \cdot \lambda\right)\right)^2}{2 \left(\frac{T}{4}\right)^2}\right)
\end{equation}
which shift the peak timing of terminal reward G in \ref{e.4.8} to left as $\lambda$ increase, and  $\gamma$ is the shift magnitude such that $G^\lambda(t)$ take its maximum value at $\frac{T}{2} -\gamma T(1-\lambda)$, where $\gamma \in [0, \frac{1}{2}]$. In the following analysis, $\gamma = \frac{1}{10}$ was chosen so that the peak of $G(t)$ was shift to the left by $\frac{T}{10}$, which indicates onset of growing season is 7.2 days earlier in the model. This choice is motivated by \cite{zheng2022earlier}, which estimated that the earliest onset was advancing at 7.64 days/decade. Here $\lambda$ was used again since  deterioration level and earlier onset of growing season in Arctic can both be attributed to global warming. Following the same procedures above, we obtain the optimal control and duration of staying at each stopover site.
\begin{figure}[!htbp]
  \centering
    \begin{minipage}{0.4\textwidth} 
        \includegraphics[width=\linewidth]{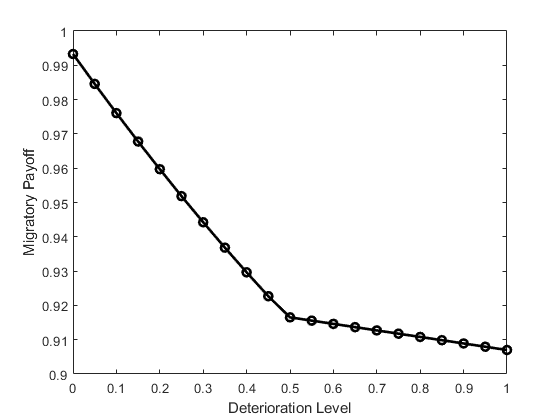}
       \caption{Migratory payoff after adjusting terminal reward}
       \label{fig:New Payoff}
    \end{minipage}
    \begin{minipage}{0.4\textwidth} 
        \includegraphics[width=\linewidth]{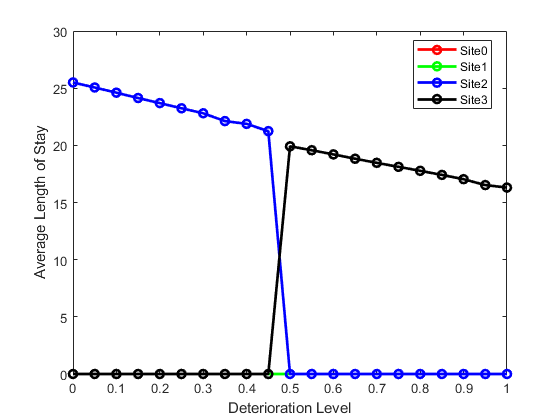}
        \caption{Average duration of stay at stopover site against $\lambda$ after effect on terminal}
        \label{fig:LOS against lamda 2}
    \end{minipage}
\end{figure}

Compared with Figure \ref{fig:V(0,0) against lamda} and \ref{fig:LOS against lamda}, one observation is that the migratory payoff continue to decrease even after stopover site 2 was abandoned and birds choose to wait at stopover site 3. This is due to that as the onset of growing season in Arctic become earlier and peak terminal reward shift left, there is less time available for birds to rest at any stopover sites, which results in less reward from resting at any stopover site and thus decrease migratory payoff even when the deteriorated stopover site was no longer utilized. Another observation is that the duration of staying at stopover sites 2 and 3 both decreased as the deterioration level increased before the critical deterioration level. This is consistent with the observation by \cite{smith2012trends} and \cite{hupp2018spring} that birds tended to stay for a shorter period of time at stopover sites due to effect of global warming.



\subsection{Perfect vs Imperfect information}
In previous sections, we have observed how individuals maximize their payoff by choosing a migration strategy, which is based on the assumption that the individual agent has perfect information of the environment and can make prediction about their journey. However, as discussed in the introduction, individual birds typically do not have access to perfect environmental information. 
While they can rely on past experience to estimate the conditions in their migratory route based on the perception, the actual environmental condition is likely subject to additional factors such as stochasticity/noise, or other changes due to human activities and global climate change.
In any case, individual birds must make decisions based on imperfect information, leading to deviations between the expected and actual rewards.

In some circumstances, the individual bird can gain access to more accurate environmental condition at the destination by utilizing a stopover site that is closer to the destination, in contrast to making a long distance direct flight. For example, the Icelandic whimbrel (\textit{Numenius phaeopus islandicus}) migrates from Iceland to West Africa during spring migration, utilizing West Europe as a stopover site \cite{gunnarsson2016migration}. Carneiro et al. \cite{carneiro2020linking} identified two primary spring migration strategies employed by this species: a direct flight from wintering to breeding sites, versus the incorporation of a stopover site. The latter strategy is enables the birds to evaluate terminal site weather conditions at the stopover, thereby reducing weather-related risks, and is generally preferred. 

Here, we attempt to address the following biological question:
\begin{itemize}
    \item How does weather dynamics at the terminal site impact bird migration when more accurate weather information can be gained when utilizing a stopover site? 
\end{itemize}



For clarity, denote by $g(t)$ the perceived terminal reward function (based on imperfect information) and denote by $G(t)$ the actual terminal reward function (based on perfect information). To address the aforementioned question, we investigate the effect of information on both the expected payoff and migration strategy, which are analyzed in two aspects corresponding to the following subsections:

In Subsection \ref{sec:5.2.1}, we explore the effect of noise in the terminal reward function. Precisely, we assume that the actual reward function is a perturbation of the perceived reward:
\begin{equation}      
G(t,x) = g(t,x) + A  \kappa(t),
\label{eq: fluctuation}
\end{equation}
where $\kappa(t)$ is an oscillatory function with high frequency and small amplitude. The parameter $A$ governs the intensity variation of the fluctuation with values ranging from 0 to 1.

In Subsection \ref{sec:5.2.2}, we are motivated by  \cite{carneiro2020linking} to explore the situation where individuals have access to the perfect information $G(t)$ by using stopover site(s). Here, we mainly discuss about two scenarios as follows:
\begin{itemize}

    \item [] {\bf Scenario1}: The perceived terminal reward function $g(t)$ is a coarse-grained version of the actual reward funtion $G(t)$, i.e. represented as a projection of $G(t)$ into the space of step functions:
    \begin{equation}
        g_n(t) = \sum^{n}_{k=0} \chi_{\left\{ t_{k} \leq t \leq t_{k+1} \right\}} \fint_{t_k}^{t_{k+1}} G\,dt, \quad \text{where } 0 = t_0 \leq  \dots \leq t_k = \frac{kT}{n} \leq \dots \leq t_n = T.
    \label{eq: mean_value}
    \end{equation}

    \item [] {\bf Scenario2}: Climate change can cause the actual weather to deviate significantly from the individual's experience or perception of the past. For instance, global warming may shift the timing and intensity of the green-up peak at the terminal site \cite{robertson2024decoupling}.
    The relationship of $g(t)$ and $G(t)$ can be written as follows. 
    \begin{equation}
        G(t) = g(t+t_{move}),
        \label{eq: moved_peak}
    \end{equation} where $t_{move}$ characterizes how much the ``green up" time arrives earlier.
\end{itemize}

We apply the numerical methods introduced in Section 4 as follows:
\begin{itemize}
    \item First, we examine the difference between expected and real optimal value at $(0,0)$ guided by optimal control under partial information. We will also see the change of the difference with strength $A$ increasing. See details in Subsection \ref{sec:5.2.1};
    \item Second, we observe how the optimal control (i.e. the collection of switching regions) changes if the agent can access (and optimize with) the perfect information $G(t)$ after spending time at the stopover site. We will discuss two specific cases as mentioned above. See details in Subsection \ref{sec:5.2.2}.
\end{itemize}


\vspace{2mm}

The parameters used in the following simulations are summarized in Table \ref{parameter table2}.

\begin{table}[!htbp!]
\centering
\begin{threeparttable}
\caption{Parameters}
\begin{tabular}{|p{3.5cm}|c|p{7cm}|}
\hline
\textbf{Parameter} & \textbf{Baseline Value} & \textbf{Description}\\ \hline
$L$ & 6450 km\tnote{a} & Distance from Iceland to Guinea in West Africa\\ \hline

$v_1$ & {277.56 km/day\tnote{b}} & Detour flight velocity\\ \hline
$v_2$ & {293.328 km/day\tnote{b}} & Direct flight velocity\\ \hline
$v_3$ & 0 & Velocity of waiting\\ \hline
$T$ & 70 days\tnote{a} & Total number of days covered\\ 
\hline
$\mu_1$ & {145 km$^2$/(day)} & Volatility associated with detour migration\\ 
\hline
$\mu_2$ & {150 km$^2$/(day)} & Volatility associated with direct migration\\ 
\hline
$\mu_3$ & 0 & Volatility associated with direct migration\\ \hline
$\beta_1$ & 0.0005 {/day} & Mortality risk for detour migration\\ \hline
$\beta_2$ & 0.00055 {/day} & Mortality risk for direct migration \\ \hline
$\beta_3$ & {0.0004 /day} & {Mortality risk for waiting at the stopover site} \\
\hline
$n$ & 1\tnote{c} & Number of stopover sites available (excluding wintering site)\\ \hline

$f_1$ & 0 & Reward for regime 1: detour flight\\\hline
$f_2$ & 0 & Reward for regime 2: direct flight\\\hline
$f_3$ & $\displaystyle \begin{cases} 
0.00165 \tnote{d} & \text{wintering site} \\
0.00276 \tnote{d}& \text{stopover site} \\
0 & \text{otherwise}
\end{cases}$ & Reward for regime 3: waiting\\\hline
$h_{ij}$ ($i,j \in \{1,2,3,4\}$)& {$\displaystyle \begin{cases} 
0.001 \tnote{e} & \text{wintering site} \\
0.05 \tnote{e}& \text{stopover site} \end{cases}$} & Switching cost\\\hline
\end{tabular}
\label{parameter table2}
\begin{tablenotes}
\item[a] {We} consider the spring migration of \textit{Numenius phaeopus islandicus} from {West Africa to Iceland}. According to \cite{carneiro2019faster}, the {migration distance} is around 6450 km with a perturbation of 118 km {over a span of} 70 days.

\item[b] According to  \cite{castro1989flight} and \cite{carneiro2019faster}, the {reasonable minimum and maximum ground speed of \textit{Numenius phaeopus islandicus} in spring are 277.56 km/day and 293.328 km/day.} As we fix the distance for both detour and direct migration, we need to distinguish the two regimes by velocity. {To this end, we take the biggest reasonable variance, assigning 277.56 km/day for detour and 293.328 km/day for direct migration.}

\item[c] {Although} according to \cite{carneiro2019faster}, \textit{Numenius phaeopus islandicus} do have several stopover sites to choose during the spring migration, like Ireland, western Britain, northwest France, and Portugal, the phenomena observed in {\cite{carneiro2020linking} only} involves one stopover site. Here, we put it at the {position} around $\frac{2}{3}L$ {, representing the western Europe as a whole}.

\item[d] {The reward is scaled consistently with Section 5.1.} 

\item[e] 
{Switching costs are uniformly set to approximately $10\%$ of the terminal reward, except outside the wintering and stopover sites, where they are set to infinity. Direct} flight is only allowed at the wintering site. After introducing regime $4$ in Subsection \ref{sec:5.2.2}, $h_{34}$ is the only {finite} switching cost at the stopover site.
\end{tablenotes}
\end{threeparttable}
\end{table}

\subsubsection{Diffusion guided by partial information}\label{sec:5.2.1}

In this subsection, we explore the impact of the noise in the terminal reward on the optimal value. To this end, we impose \eqref{eq: fluctuation} with 
$$g(t) = \exp \bigg(-\frac{(t-\frac{3}{4}T)^2}{2(\frac{T}{8})^2}\bigg) \quad \text{and} \quad \kappa(t) = 0.1 \sin \bigg(\frac{140 \pi}{T} t\bigg).$$
We will compute two value functions. One is the value function $V^{imperfect}$ which the individual agent uses to optimize its strategy. This can be computed using the Hamilton-Jacobi-Bellman equations, and can be interpreted as the perceived reward by the individual. The other value function is the actual value function $V^{perfect}$ that takes into account that the actual terminal reward $G(t)$ is different from the perceived reward $g(t)$ that the individual agent used in the optimization process. Hence $V^{perfect}  - V^{imperfect}$ represents the effect of information in and the mismatch between the perceived expected payoff and the actual expected payoff.  We will study the relationship between the amplitude of the fluctuation $A$ and the difference of value functions $D:=V^{perfect}  - V^{imperfect}$ at $(0,0)$ (conditioned on the individual starting at $(t,x) = (0,0)$). 

First, we perform the optimization of the individual agent based on the perceived terminal reward function $g(t)$, i.e. 
following the steps outlined in Subsection 4.4 to compute the optimal control $\alpha_{imperfect}$  by solving for the solutions $V_i(t,x)$ of the PDE systems \eqref{e.4.8} with $G(t)$ replaced by $g(t)$, and identifying the corresponding switching regions $S_{ij}$. Note that it is enough show the switching regions conditioned on the individual being states 1 (detour flight) or 3 (waiting) since the individual stays in state 2 (direct flight) once it adopts direct flight. The switching region from states $1$ to $\{2,3\}$ and from $3$ to $\{1,2\}$ are shown respectively in  Figure \ref{fig:regime1(region)_imperfect(case1)} and Figure \ref{fig:regime3_imperfect(case1)}. Hereafter we call these switching regions (which is calculated based on imperfect information) $\alpha_{imperfect}$. It is demonstrated that {most birds prefer to choose detour flight at first. After arriving at the stopover site, they tend to switch to waiting state and wait for the optimal time (approximately $t=40$) to continue their flight and finish their migration}. 

\begin{figure}[ht]
    \centering
    \begin{subfigure}[b]{0.4\textwidth}
        \includegraphics[width=\textwidth]{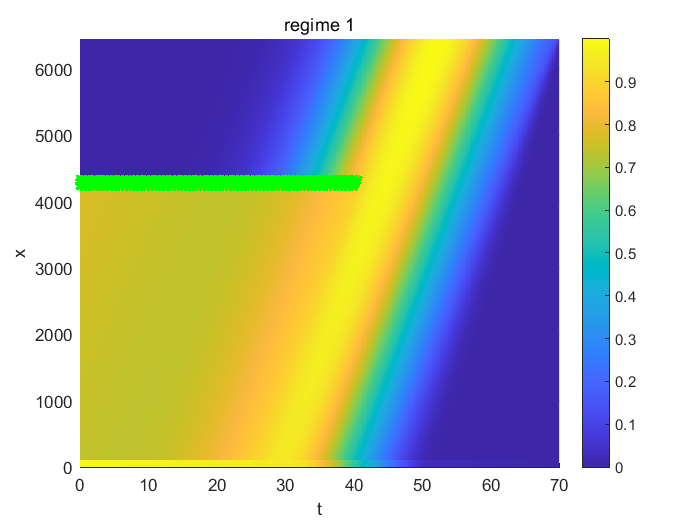}
        \caption{Regime 1: detour flight}
        \label{fig:regime1(region)_imperfect(case1)}
    \end{subfigure}
    \begin{subfigure}[b]{0.4\textwidth}
        \includegraphics[width=\textwidth]{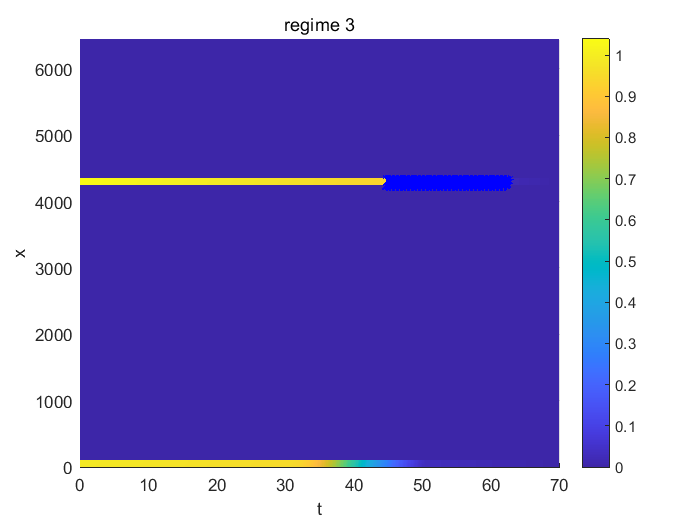}
        \caption{Regime 3: waiting}
        \label{fig:regime3_imperfect(case1)}
    \end{subfigure}
    \vfill
    \begin{subfigure}[b]{0.4\textwidth}
        \includegraphics[width=\textwidth]{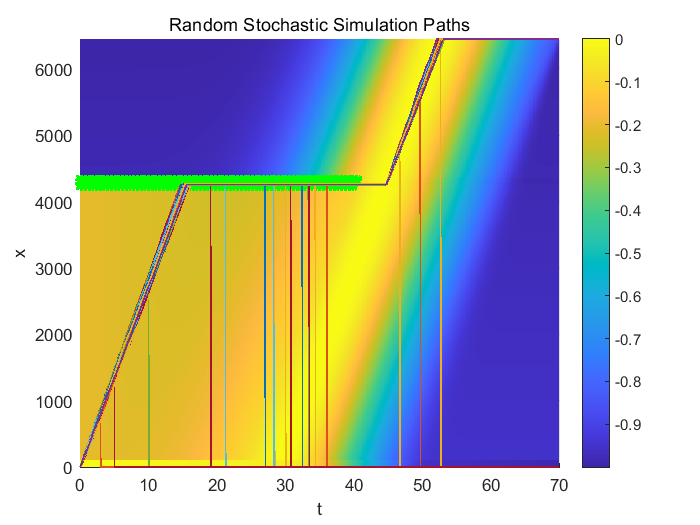}
        \caption{Stochastic simulation}
        \label{fig:regime1(simulation)_imperfect(case1)}
    \end{subfigure}
    \begin{subfigure}[b]{0.4\textwidth}
        \includegraphics[width=\textwidth]{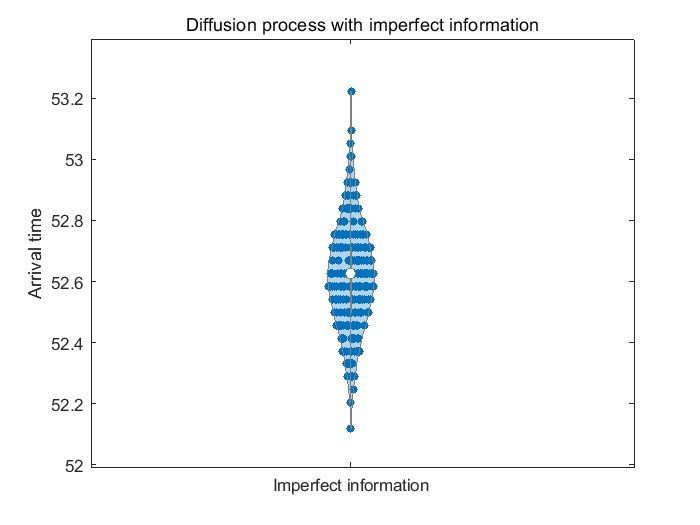}
        \caption{Distribution of arrival time}
        \label{fig:arrival_time_distribution(case1)}
    \end{subfigure}
    \caption{Result under imperfect information}
    \label{fig:Result under imperfect information}
\end{figure}

Next, we perform Monte Carlo simulation of the stochastic diffusion process with the control $\alpha_{imperfect}$, repeated $N = 500$ times, to produce a family of possible migrating routes
(see in Figure \ref{fig:regime1(simulation)_imperfect(case1)}). From these simulations, we obtain the statistics of  arrival time at $x = L$ based on control $\alpha_{imperfect}$, as shown in the violinplot in Figure \ref{fig:arrival_time_distribution(case1)}. It is observed that this distribution qualitatively aligns with the shape of the perceived terminal green-up timing $g(t)$. {As depicted in the figure, most birds choose to arrive at terminal site around the peak of green-up time to accept better terminal reward.} This shows that individual's experience and perception, embedded in partial information, determine the individuals' decision-making to a large extent.

Then, we estimate the difference of optimal value at $(0,0)$ under perfect and imperfect information. We first use the following equation to calculate $V^{imperfect}(0,0)$:
\begin{equation}
\begin{split}
    &V^{imperfect}(0,0) \\
    &\ \ = E^{\alpha_{imperfect}} \bigg[ \int_{0}^{\tau \wedge T} e^{-\beta_{i}(s, 0) s} f_{i}(s,X_{s}^{0,0,i}) \, ds -\sum_{\tau_n \leqslant \tau} e^{-\beta_{i}(\tau_{n}, 0) \tau_{n}} h_{l_{n-1},l_{n}} + \chi_{\left\{ \tau < T \right\}} e^{-\beta_{i}(\tau, 0) \tau} g(\tau) \bigg].
    \label{eq: simulation_partial}
\end{split}
\end{equation}
However, in reality, bird is led by partial information but gets actual reward as payoff, which is given by the following equation:
\begin{equation}
\begin{split}
    &V^{perfect}(0,0)\\
    &\ \ = E^{\alpha_{imperfect}} \bigg[ \int_{0}^{\tau \wedge T} e^{-\beta_{i}(s, 0) s} f_{i}(s,X_{s}^{0,0,i}) \, ds -\sum_{\tau_n \leqslant \tau} e^{-\beta_{i}(\tau_{n}, 0) \tau_{n}} h_{l_{n-1},l_{n}} + \chi_{\left\{ \tau < T \right\}} e^{-\beta_{i}(\tau, 0) \tau} G(\tau) \bigg].
    \label{eq: simulation_perfect}
\end{split}
\end{equation}
Therefore, with the distribution obtained above, we can compute the expectation of difference value of value function at $(0,0)$, denoted as $D$, with equation \eqref{eq: simulation_partial} and \eqref{eq: simulation_perfect}.
\begin{equation}
    D := V^{imperfect}(0,0) - V^{perfect}(0,0) = \frac{1}{N}\sum^{N}_{k = 1}  e^{-\beta_{i}(\tau_k, 0) \tau_k}(g(\tau_k)-G(\tau_k)) .
\end{equation}

Then we proceed to test the influence of the amplitude $A$ on the variance of the difference over all the stochastic simulations. Denote the variance as $Var$, whose calculation formula is
\begin{equation}
    Var := \frac{1}{N} \sum^{N}_{k=1} (e^{-\beta_{i}(\tau_k, 0) \tau_k}(g(\tau_k)-G(\tau_k))-D)^2.
\end{equation}

We observed that with the amplitude $A$ increasing, differences of optimal value will exhibit more variability. The result is consistent with our speculation, as shown in Figure \ref{fig:Variance_and_Amplitude(case1)}. From a biological point, this shows that with more intensive noise mixed in the terminal weather conditions, the difference of birds' expectation and reality will show extreme variability, thus increasing the volatility of getting different payoff with birds' expectation. Therefore, under this circumstance, a stopover site for obtaining perfect information becomes necessary to reduce the risk of mismatch.
\begin{figure}[htbp!]
    \centering
    \begin{minipage}{0.45\textwidth} 
        \includegraphics[width=\linewidth]{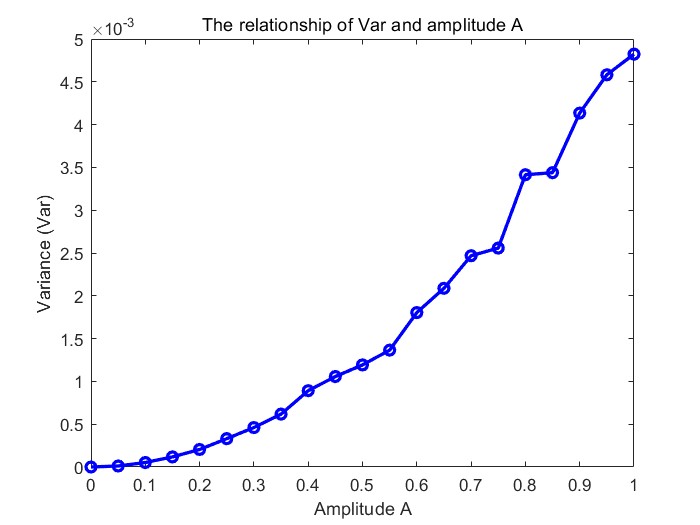}
    \caption{Variance value $Var$ under different amplitude $A$}
    \label{fig:Variance_and_Amplitude(case1)}
    \end{minipage}
    \hfill 
    \begin{minipage}{0.45\textwidth} 
        \includegraphics[width=\linewidth]{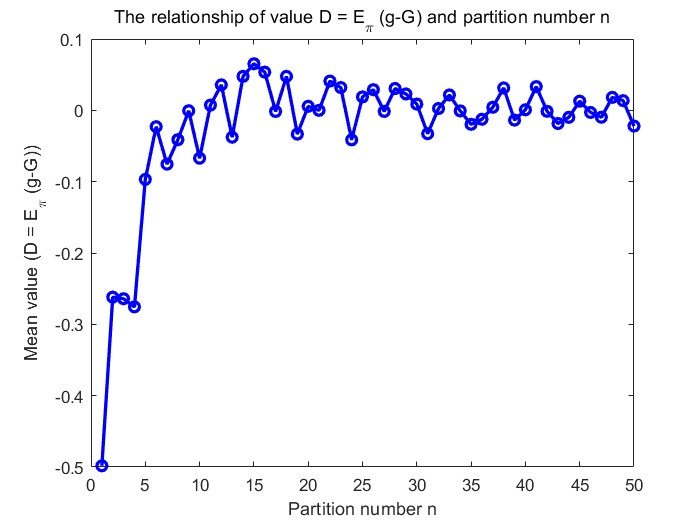}
    \caption{Mean value D under different partition number n}
    \label{fig:Mean_value_with_partition_n}
    \end{minipage}
\end{figure}
\subsubsection{When individual can gain information after using stopover site}\label{sec:5.2.2}
In this subsection, we propose a new method to  explore the effect of information on  the optimal switching strategy changes. Specifically, we will model an individual as entering a new state with  perfect information (regarding the quality of terminal site) if it chooses to rest at a particular stopover site. 

The numerical method for this case is outlined as follows. To address the state of perceiving perfect information after waiting at the stopover site, we introduce a new regime 4, marked in color magenta in graph, to the PDE system with its terminal reward replaced by perfect information $G(t)$. The parameters $v_4, \mu_4, \beta_4$ of regime 4 are the same with regime 1 (i.e., detour flight). The individual who gained information cannot lose information anymore, i.e. for switching costs, we only set $h_{34} = 0.05$ at the stopover site and set any other switching costs related to regime $4$ at any other places to $\infty$.
In this updated system, we use the same numerical approach described in Section 4 to compute the optimal control. Subsequently, we perform stochastic simulations. Initially, the bird follows the control derived under imperfect information, which corresponds to regimes 1, 2, and 3. If the bird chooses to wait at the stopover site, it switches to Regime 4 and follows the control derived under perfect information thereafter. This framework effectively implements the process of information-gathering at the stopover site.

We recall that $g(t)$ is the projection of $G(t)$ into step functions space in {\bf Scenario1}. In {\bf Scenario2}, $G(t)$ has the same shape as $g(t)$ but with the peak arriving earlier, as defined in \eqref{eq: mean_value} and \eqref{eq: moved_peak}.

\vspace{3mm}
\noindent
{\bf \textbf{Scenario1:} Projection of $G(t)$ onto step functions spaces}

We first focus on {\bf Scenario1}. Fix the interval parameter $n = 2$ and choose  
$$G(t) = \begin{cases} 
\frac{4}{T} (t-(\frac{3}{4}T-\frac{T}{4})), & \text{if } x \in [\frac{1}{2}T, \frac{3}{4}T], \\
-\frac{4}{T} (t-(\frac{3}{4}T+\frac{T}{4})), & \text{if } x \in [\frac{3}{4}T, T],\\
0, & \text{otherwise} ,
\end{cases}$$
which is a single-peak function. Then when $n=2$, function $g(t) = \frac{1}{2}$ is a constant within the interval $[\frac{1}{2}T, T]$ and 0 elsewhere. For the new PDE system, we set partial information $g(t)$ as the terminal reward for regime 1,2,3 and perfect information $G(t)$ as the terminal reward for regime 4 to simulate the process of perceiving perfect information at the stopover site. By solving the solution to the HJB system, we can get the optimal control $\alpha_{perfect}$, under which we can perform further stochastic simulations. 

\begin{figure}[ht]
    \centering
    \begin{subfigure}[b]{0.4\textwidth}
        \includegraphics[width=\textwidth]{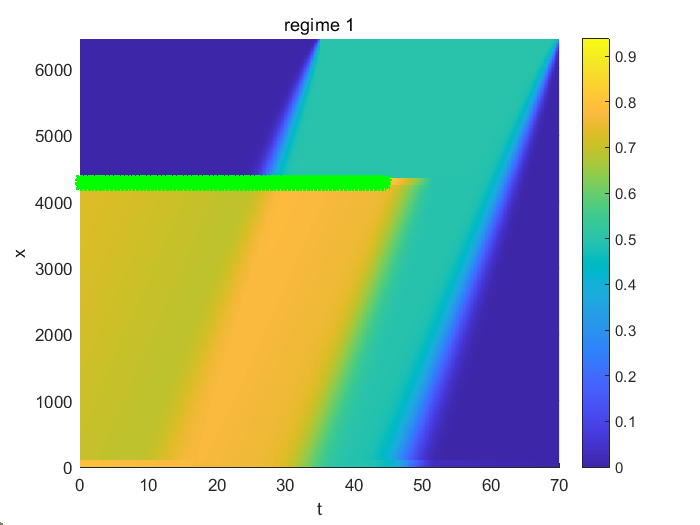}
        \caption{Regime 1: detour flight}
        \label{fig:regime1(case2_mode1)}
    \end{subfigure}
    \begin{subfigure}[b]{0.4\textwidth}
        \includegraphics[width=\textwidth]{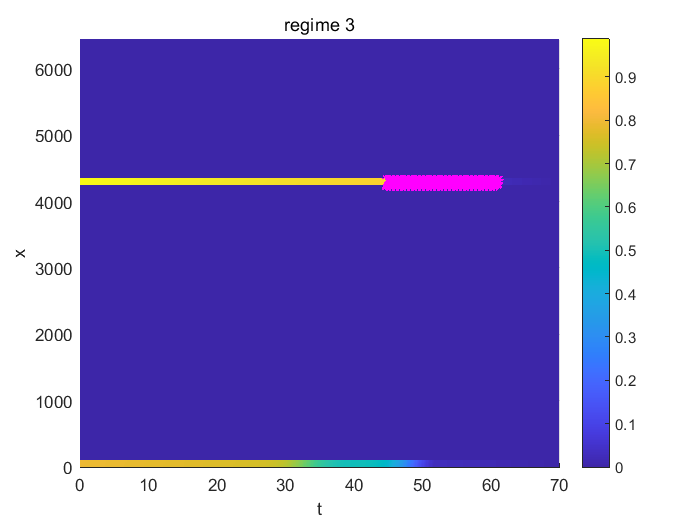}
        \caption{Regime 3: waiting}
        \label{fig:regime3(case2_mode1)}
    \end{subfigure}
    \vfill
    \begin{subfigure}[b]{0.4\textwidth}
        \includegraphics[width=\textwidth]{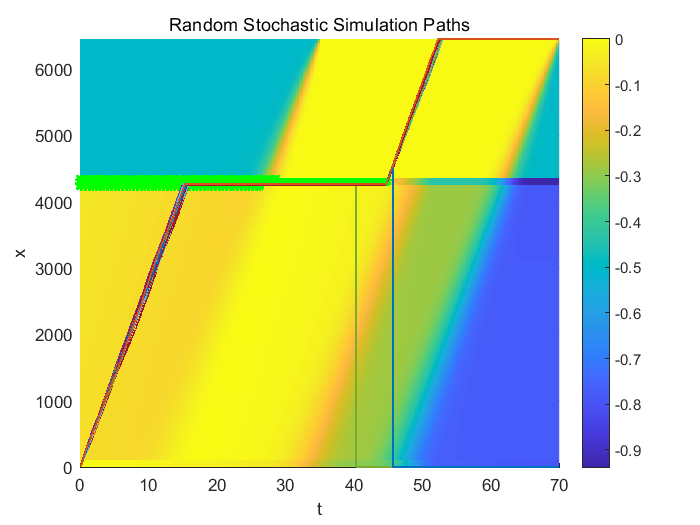}
        \caption{Stochastic simulation}
        \label{fig:regime1_simulation(case2_mode1)}
    \end{subfigure}
    \begin{subfigure}[b]{0.4\textwidth}
        \includegraphics[width=\textwidth]{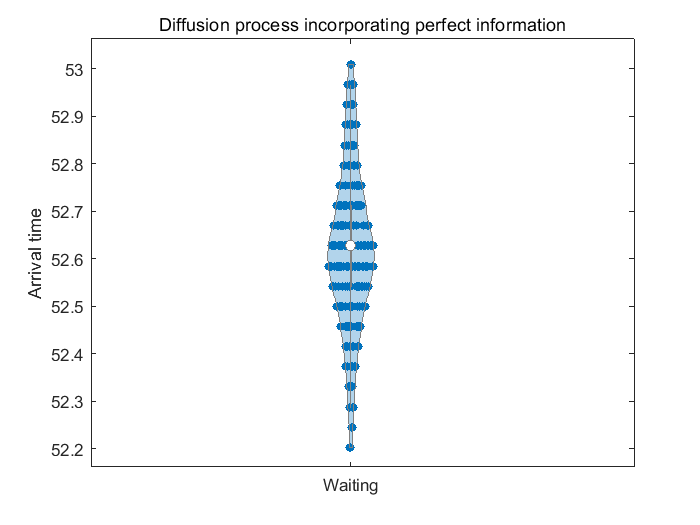}
    \caption{Arrival time distribution}
    \label{fig:arrival_time_distribution(case2_mode1)}
    \end{subfigure}
    \caption{Result incorporating perfect information (Scenario 1)}
    \label{fig:Result under imperfect information(case2mode1)}
\end{figure}
The switching region corresponding to the control $\alpha_{perfect}$ is shown in Figure \ref{fig:regime1(case2_mode1)} and Figure \ref{fig:regime3(case2_mode1)}. Combined with the simulation results in Figure \ref{fig:regime1_simulation(case2_mode1)}, it is demonstrated that birds prefer to choose to wait at the stopover site to gain an access to perfect information as it in an increased overall payoff.


The statistics of arrival time at $x = L$ is recorded as shown in the violinplot in Figure \ref{fig:arrival_time_distribution(case2_mode1)}. The shape of the distribution of arrival time aligns exactly with the shape of perfect information $G(t)$, {or otherwise, birds may choose to arrive as soon as possible due to the constant reward at the terminal site, causing a mismatch with the actual peak of green-up time of terminal reward. This shows} perceiving perfect information at the stopover site improves the individual's expected payoff by decreasing bird's risk of mismatching the green-up timing at the terminal site.

Next, we proceed to see how it will influence the expected difference of optimal value with the increase of the partition number $n$. We observe that with the increase of partition number $n$, the difference of optimal value will gradually be close to 0, as shown in Figure \ref{fig:Mean_value_with_partition_n}. This result aligns with our expectation because as $n$ increases, the partial information $g(t)$ gradually converges to perfect information $G(t)$ under certain norm. 
Consequently, this convergence naturally leads to a smaller discrepancy in the expected optimal value.


In summary, in \textbf{Scenario1}, i.e.
$$
g_n(t) = \sum^{n}_{k=0} \chi_{\left\{ t_{k} \leq t \leq t_{k+1} \right\}} \fint_{t_k}^{t_{k+1}} G\,dt, 
$$
individual bird prefers to switch to waiting at the stopover site to avoid the risk of mismatch. The improvement of expected payoff is greatest when $n=0$, i.e.  the perceived terminal reward satisfies $g(t) = \fint G\,dt$. However, as 
Additionally, the expectation of difference of optimal value gradually approaches $0$ as the $n \to \infty$, so that the perceived terminal reward approaches the actual reward $g_n(t) \to G(t)$ almost everywhere.

\vspace{3mm}
\noindent
{\bf \textbf{Scenario2:} Influence of global warming}

We take 
$$g(t) = \exp \bigg(-\frac{(t-\frac{3}{4}T)^2}{2(\frac{T}{8})^2}\bigg) \ \ \text{and} \ \ G(t) = \exp \bigg(-\frac{(t-\frac{1}{2}T)^2}{2(\frac{T}{8})^2}\bigg),$$
where the peak green-up timing is shifted earlier by $t_{move} = \frac{T}{4}$ due to the global warming effect. 

\begin{figure}[ht]
    \centering
    \begin{subfigure}[b]{0.4\textwidth}
        \includegraphics[width=\textwidth]{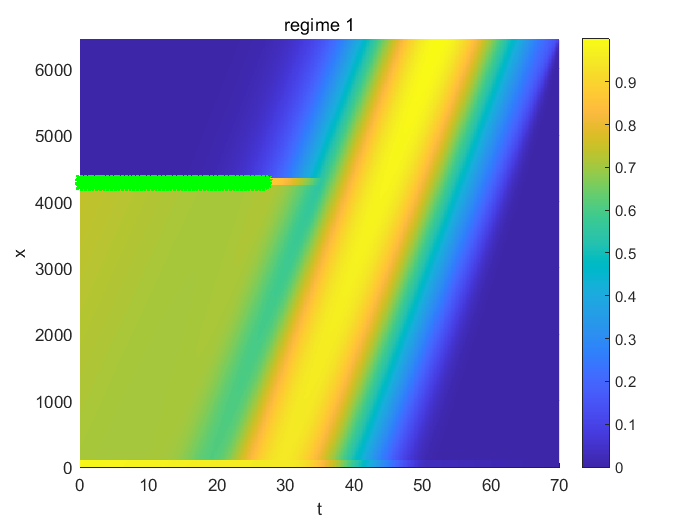}
        \caption{Regime 1: detour flight}
        \label{fig:regime1(case2_mode2)}
    \end{subfigure}
    \begin{subfigure}[b]{0.4\textwidth}
        \includegraphics[width=\textwidth]{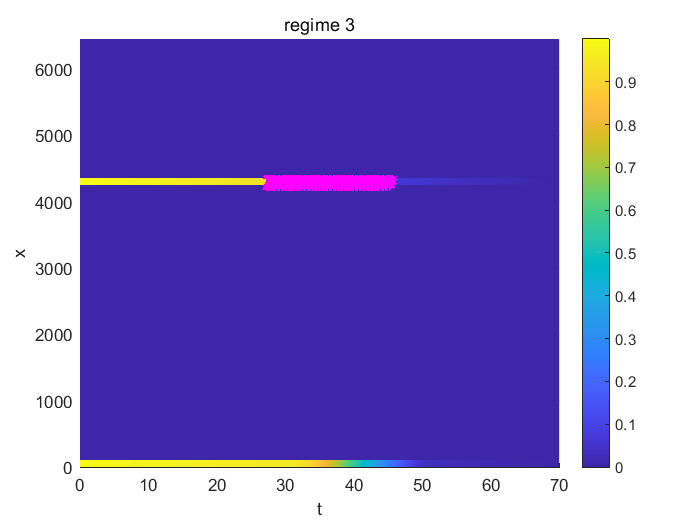}
        \caption{Regime 3: waiting}
        \label{fig:regime3(case2_mode2)}
    \end{subfigure}
    \vfill
    \begin{subfigure}[b]{0.4\textwidth}
        \includegraphics[width=\textwidth]{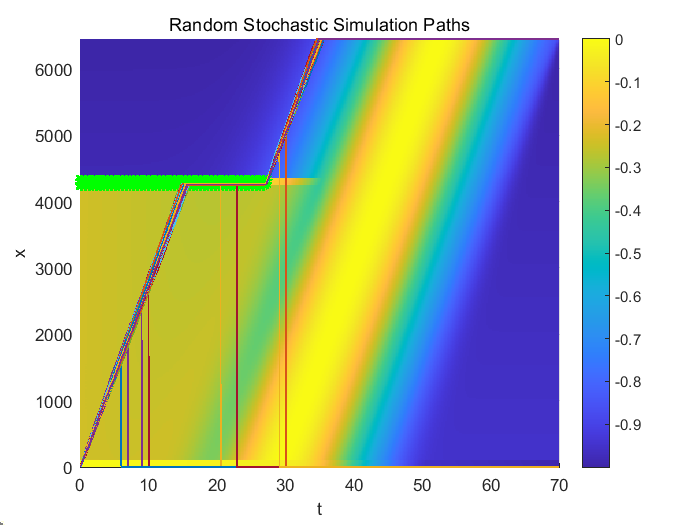}
        \caption{Stochastic simulation}
        \label{fig:regime1_simulation(case2_mode2)}
    \end{subfigure}
    \begin{subfigure}[b]{0.4\textwidth}
        \includegraphics[width=\textwidth]{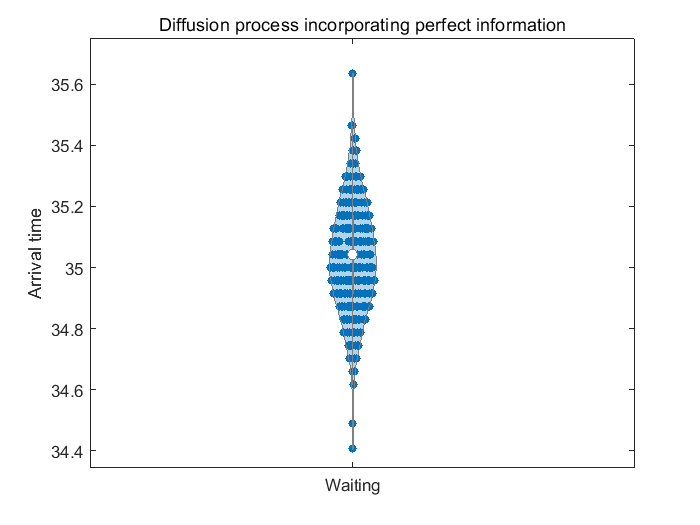}
    \caption{Comparison of arrival time }
    \label{fig:arrival_time_distribution(case2_mode2)}
    \end{subfigure}
    \caption{{Result incorporating perfect information (Scenario2)}}
    \label{fig:Result under imperfect information(case2mode2)}
\end{figure}
With the same PDE systems described in \textbf{Scenario1}, we can compute the optimal control $\alpha_{perfect}$ and its corresponding switching region, as shown in Figure \ref{fig:regime1(case2_mode2)} and Figure \ref{fig:regime3(case2_mode2)}. Then we move on to simulate the diffusion process, with the result shown in Figure \ref{fig:regime1_simulation(case2_mode2)}.

We observe that birds changes their behavior after they obtain more accurate information of the terminal reward function. Indeed, the optimal switching control $\alpha_{imperfect}$ aims at allowing individual to arrive at $x=L$ around the {\it perceived peak green-up timing} which is $t=\tfrac34 T$. This is illustrated in Figure \ref{fig:Result under imperfect information}. In contrast, individuals which gained access to better information at the stopover site deviates the optimal control $\alpha_{imperfect}$ by setting off earlier and aims at reaching the terminal site at $t=\tfrac12 T$.

{To show the difference of arrival time more clearly, we record the arrival time for the individuals shown in Figure \ref{fig:regime1_simulation(case2_mode2)} in the statistical way as before. The statistics of their arrival time is shown in Figure \ref{fig:arrival_time_distribution(case2_mode2)}, which perfectly match with our optimal control shown in Figure \ref{fig:regime1(case2_mode2)} and \ref{fig:regime3(case2_mode2)}, showing the influence of perceiving perfect information at the stopover site. With the violinplot shown in Figure \ref{fig:arrival_time_distribution(case2_mode2)}, we can observe that birds waiting at the stopover site changes their strategies and try to reach the terminal site at around $\frac{1}{2}T$, aligning with the peak time of actual terminal reward. By contrast, as shown in Figure \ref{fig:arrival_time_distribution(case1)}, birds having waited at the stopover site mostly arrived at $\frac{3}{4}T$, showing they are guided by partial information and thus would end with an intensive mismatch with the actual terminal reward.}

To summarize our result for \textbf{Scenario 2}, it is demonstrated that as a plausible response to the increased mismatch between perceived and actual environmental information (due to, e.g. increased stochasticity of the global climate and/or global warming), individual may utilize certain key stopover site to gain a better estimation of the weather condition at the terminal site, which can in turn help them decrease the mismatch in the timing of arrival with the peak green-up time at the terminal site.

In conclusion, under global climate change, it is not enough for bird population to rely solely on past experience or perception in their selection of migratory strategies. In fact, 
stopover sites play an increasingly important role for birds in interpreting the weather conditions along their migratory route, helping them to better select migratory strategies. Therefore, it is even more important to protect and preserve these key stopover locations to ensure the survival of migratory bird populations.

\end{document}